\documentclass[11pt]{amsart}

\usepackage[hmargin=0.8in,height=8.6in]{geometry}
\usepackage{amssymb,amsthm, times, enumerate}
\usepackage{delarray,verbatim}
\usepackage{mathrsfs}
\usepackage[dvipsnames]{xcolor}

\usepackage{ifpdf}
\usepackage{color}
\definecolor{webgreen}{rgb}{0,.5,0}
\definecolor{webbrown}{rgb}{.8,0,0}
\definecolor{emphcolor}{rgb}{0.5,0.95,0.95}

\usepackage{hyperref}
\hypersetup{%
	colorlinks=true,
	linkcolor=webbrown,
	filecolor=webbrown,
	citecolor=webgreen,
	breaklinks=true}
\ifpdf \hypersetup{pdftex,
	pdfstartview=FitH, 
	bookmarksopen=true,
	bookmarksnumbered=true
} \else \hypersetup{dvips} \fi

\newcommand{\red}{\textcolor[rgb]{1.00,0.00,0.00}}

\newcommand{\green}{\textcolor[rgb]{0.00,0.70,0.00}}
\usepackage{enumerate}
\usepackage{amsmath,amsthm,amsfonts}
\usepackage{hyperref}
\usepackage{mathtools}
\usepackage[utf8]{inputenc}
\usepackage{footnote}

\makesavenoteenv{table}
\makesavenoteenv{tabular}
\usepackage[titletoc]{appendix}
\newtheorem{theo}{Theorem}
\newtheorem{mydef}{Definition}
\newtheorem{lemma}{Lemma}
\newtheorem{prop}{Proposition}
\newtheorem{remark}{Remark}
\newtheorem{assum}{Assumption}
\newcommand{\E}{\mathbb{E}}
\newcommand{\Eq}{\mathcal{E}}
\newcommand{\D}{\mathcal{D}}

\newcommand{\R}{\mathbb{R}}

\newcommand{\Prob}{\mathbb{P}}

\newcommand{\T}{\mathcal{T}}

\newcommand{\V}{\mathcal{V}}

\newcommand{\Ind}{\mathbb{I}}

\newcommand{\lev}{L\'{e}vy }

\begin{document}
	\title[Effects of positive jumps of assets on endogenous bankruptcy and optimal capital structure]{Effects of positive jumps of assets on endogenous bankruptcy and optimal capital structure: Continuous- and periodic-observation models}
	
	\author[D. Mata]{Dante Mata L\'opez$^*$}

\author[J.L. P\'erez]{Jos\'e Luis P\'erez$^*$}\thanks{$^*$Department of Probability and Statistics, Centro de Investigaci\'on en Matem\'aticas, A.C. Calle Jalisco S/N
C.P. 36240, Guanajuato, Mexico,
email: dante.mata@cimat.mx, jluis.garmendia@cimat.mx}

\author[K. Yamazaki]{Kazutoshi Yamazaki$^\dagger$}\thanks{$^\dagger$Department of Mathematics,
Faculty of Engineering Science, Kansai University, 3-3-35 Yamate-cho, Suita-shi, Osaka 564-8680, Japan,
email: kyamazak@kansai-u.ac.jp}
\thanks{ K. Yamazaki is partially supported by MEXT KAKENHI grant no.\ 19H01791 and 	20K03758.}

\maketitle

\begin{abstract} 


In this paper, we study the optimal capital structure model with
endogenous bankruptcy when the firm's asset value follows an exponential \lev process with positive jumps. In the Leland-Toft framework \cite{LelandToft96}, we obtain the optimal bankruptcy barrier in the classical continuous-observation model and the periodic-observation model, recently studied by Palmowski et al.\ \cite{palmowski2019leland}. We further consider the two-stage optimization problem of obtaining the optimal capital structure.  Detailed numerical experiments are conducted to study the sensitivity of the firm's decision-making with respect to the observation frequency and positive jumps of the asset value.

\noindent \small{\noindent  {AMS 2020} Subject Classifications: 60G40,  60G51,  91G40 
}
\end{abstract}


\section{Introduction}

The classical Leland model \cite{Leland94} studies the decision-making faced by a firm on the determination of the time of bankruptcy and capital structures. The model features the 
\emph{endogenous bankruptcy} determined to solve the trade-off between maximizing the tax benefits and minimizing the bankruptcy costs. The Leland-Toft model \cite{LelandToft96}  is an extension of the Leland model  \cite{Leland94}, that successfully avoids the use of perpetual bonds by considering a particular debt profile.  It is regarded as one of the most important models in the fields of both corporate finance and credit risk. 

%

In this paper, we focus on the  effects of positive jumps of the firm's asset value  by considering a spectrally positive \lev process (
a \lev process with only positive jumps) in the Leland-Toft model.
In the past, the classical geometric Brownian motion model has been generalized successfully to several exponential \lev models. 
With the motivation of avoiding an undesirable conclusion drawn in the geometric Brownian motion model that the credit spread approaches zero as the maturity decreases, a majority of papers have focused on analyzing the effects of negative jumps of the firm's asset price process (see Hilberink and Rogers \cite{HilbRog02}, Kyprianou and Surya \cite{Kyprianou2007}, and Surya and Yamazaki \cite{surya2014optimal}). On the other hand, as discussed by Chen and Kou \cite{ChenKou09}, who considered the double jump diffusion (with i.i.d.\ exponential jumps in both directions), positive jumps in the asset value also have significant effects on the optimal capital structure and the credit spread. In this paper, we revisit the study of analyzing the influence of positive jumps by incorporating several features that are not considered in \cite{ChenKou09} and other related studies.
%
%

We consider the classical continuous-observation and periodic-observation models.
 In the continuous-observation model, the bankruptcy time is modeled by the first time the firm's asset value
$(V_t)_{t \geq 0}$, 
 modeled by an exponential spectrally positive \lev process, 
 goes below a barrier $V_B$:
\begin{align}
\inf \{ t > 0: V_{t} < V_B \}. \label{our_default_continuous}
\end{align}
The periodic-observation model, recently studied by Palmowski et al.\ \cite{palmowski2019leland} considers the scenario in which the asset value information is updated only at epochs 
$\mathcal{T} = (T_n^\lambda)_{n \geq 1}$, 
 given by the jump times of an independent Poisson process $N^\lambda = (N^\lambda_t)_{t \geq 0}$ with a fixed rate $\lambda$.  The bankruptcy time is modeled as the first observation time at which the asset value process is below $V_B$, namely, 
\begin{align*}
\inf \{ T^\lambda \in \mathcal{T}: V_{T^\lambda} < V_B \}. 
\end{align*}
As noted in \cite{palmowski2019leland}, this is also written as the classical bankruptcy time \eqref{our_default_continuous} with $(V_t)_{t \geq 0}$  replaced by  the asset value if it is only updated at 
$\mathcal{T}$, namely,
\begin{align}
V^\lambda_t := V_{T^\lambda_{N^\lambda_t}}, \quad t \geq 0, \label{V_lambda}
\end{align}
where $T^\lambda_{N^\lambda_t}$ is the last observation time before $t$.
In Figure \ref{plot_simulated}, we plot sample paths of $(V_t)_{t \geq 0}$,  $(V_t^\lambda)_{t \geq 0}$, $T^\lambda$ and the corresponding bankruptcy time to illustrate these. We refer the reader to \cite{palmowski2019leland} for the detailed discussions on this model, in particular, regarding the connections with the reduced-form model and other models such as those by Duffie and Lando \cite{duffie2001term} and Fran{\c{c}}ois and Morellec  \cite{franccois2004capital}. See also  \cite{avanzi2013periodic,avanzi2014optimal,dupuis2002optimal, noba2018optimal, perez2018optimal,perez2018american} for optimal stopping and other stochastic control problems under Poisson observations.

 \begin{figure}[htbp]
\begin{center}
\begin{minipage}{1.0\textwidth}
\centering
\begin{tabular}{c}
 \includegraphics[scale=0.5]{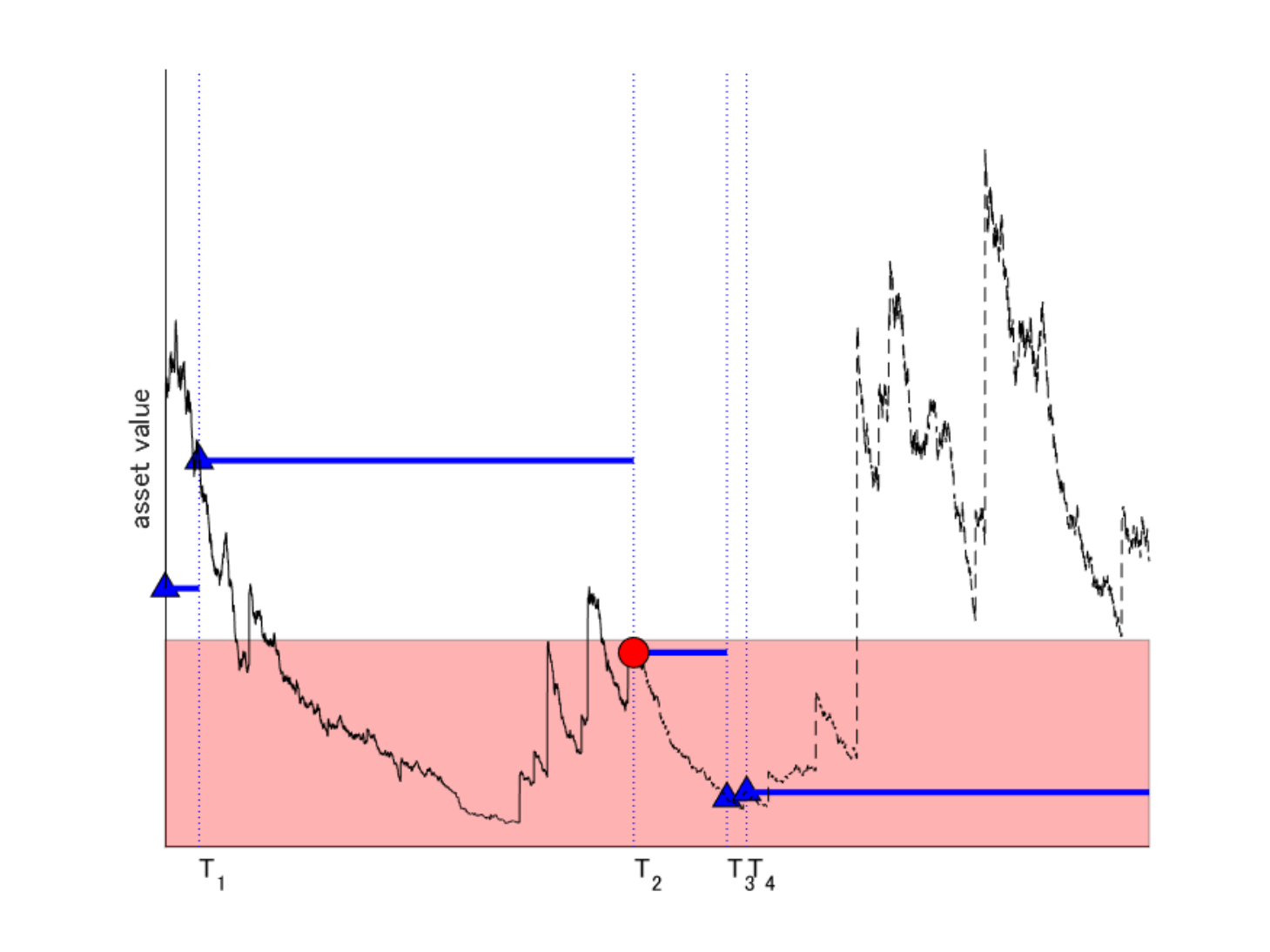}
\end{tabular}
\caption{Illustration of bankruptcy in the periodic-observation case, with sample paths of $(V_t)_{t \geq 0}$ (black lines) and $(V_t^\lambda)_{t \geq 0}$ as in \eqref{V_lambda}  (horizontal blue lines) together with the observation times $T^\lambda$  \red{$\mathcal{T}$?} (indicated by dotted vertical lines). The red rectangular region shows the bankruptcy zone $(0, V_B)$.
The asset value at  bankruptcy is indicated by the red circle and those at other observation times are indicated by blue triangles.  Here, bankruptcy occurs at $T_2^\lambda$, but the asset value has crossed $V_B$ before and then recovered back before $T_2^\lambda$.
It is also noted that the process $(V_t^\lambda)_{t \geq 0}$ jumps downward at $T_2^\lambda$ and $T_3^\lambda$ (while the original asset value $(V_t)_{t \geq 0}$ only contains positive jumps).
}  \label{plot_simulated}
\end{minipage}
\end{center}
\end{figure}

 In the periodic-observation model, as illustrated in Figure \ref{plot_simulated}  (see in particular the interval between $T_1^\lambda$ and $T_2^\lambda$), the asset value can stay below $V_B$ between the observation times. It is of  particular interest to study the event that the asset value recovers from the ``bankruptcy zone'' $(0,V_B)$ to the ``non-bankruptcy zone'' $[V_B, \infty)$  before the next observation time, and how the prospect of these events has effects on the firm's decision-making. In the case of  \cite{palmowski2019leland}, where they focus on the spectrally negative case, this event can happen only by continuous increments. Its probability is thus significantly underestimated due to the assumption of no positive jumps.  When positive jumps are allowed, the firm is expected to have a strong possibility for recovery. Thus, it is important to consider an asset value with positive jumps and analyze how these affect the firm's decision-making.

 

We consider a general \lev process with positive jumps of any arbitrary distribution. Chen and Kou \cite{ChenKou09}  studied the case with exponential jumps in both directions. However, the exponential random variable has distinctive features that other random variables do not have.
For example, as studied by  Kou and Wang \cite{kou2003first}, double exponential diffusion admits the \emph{memoryless} property that the distribution of the overshoot at up- or down-crossing times is conditionally the same as that of the original jumps.
It is also in the class of distributions with \emph{completely monotone} density, which are known to possess certain special properties.  In  \cite{palmowski2019leland}, in the two-stage optimization problem (described subsequently), a completely monotone assumption on negative jumps was imposed to show the convexity of the firm value with respect to the leverage.  Completely monotone assumptions of the \lev measure are often assumed for the optimality of a barrier strategy in other stochastic control problems (see, e.g., Loeffen \cite{loeffen2008optimality} for the optimal dividend problem). For these reasons, it is important to consider jumps with arbitrary distributions to analyze whether the same results hold as in \cite{ChenKou09}. Another feature we consider in this paper, which was not considered in \cite{ChenKou09},  is the tax-benefit threshold. To avoid the overestimation of the tax benefits, as suggested in  
Hilberink and Rogers \cite{HilbRog02}, we consider the cases in which tax benefits can be enjoyed only if the asset value is above a given threshold $V_T$.

%
%
%
%

To solve both continuous-observation and periodic-observation cases, we use the \emph{fluctuation theory} of \lev processes. For the periodic-observation case, we use the recent results by Albrecher et al.\ \cite{albrecher2016} and Landriault et al. \cite{LANDRIAULT2018152}.
The firm/debt/equity values are expressed in terms of the so-called \emph{scale 
function}, which is defined for a general spectrally one-sided \lev process.  
In the continuous-observation model, the bankruptcy level satisfies the \emph{smooth fit} condition. On the other hand, in the periodic-observation model, the bankruptcy level satisfies the \emph{continuous fit} condition. The optimal capital structure is obtained by solving the \emph{two-stage optimization problem} as proposed in \cite{LelandToft96}.  We show the uniqueness of the optimal leverage for a general spectrally positive \lev process (including the cases in which the jump size distribution does not have a completely monotone density). 

Using these analytical results, we conduct a series of numerical experiments using the case driven by a mixture of Brownian motion and i.i.d.\ positive (folded) normal distributed jumps, whose scale function is approximated by fitting phase-type \lev processes that admit an explicit form of the scale function. We verify the optimality of a selected bankruptcy level and then study the effects of the rate of observations and positive jumps.  For the former, we confirm the monotonicity of the optimal bankruptcy level $V_B^*$ and its convergence to that in the continuous-observation case.  For the latter, as conjectured above, positive jumps have significant effects on the optimal strategies.   Interestingly, our numerical results reveal that $V_B^*$ is \emph{not}  monotone in the parameter of the jumps.

%

While we focus on the one-sided jump case, as discussed in  \cite{palmowski2019leland}, the periodic-observation case can be considered to be a case driven by processes with two-sided jumps (due to the equivalence of the bankruptcy time with the classical bankruptcy time \eqref{our_default_continuous} of the process $(V_t^\lambda)_{t \geq 0}$; see also Figure \ref{plot_simulated}).
While only a few models featuring  asset value processes with two-sided jumps exist, we provide a new analytically tractable case for $(V_t^\lambda)_{t \geq 0}$, that contains two-sided jumps even when  $(V_t)_{t \geq 0}$ does not have positive jumps.  By suitably selecting the process $(V_t)_{t \geq 0}$ and $\lambda$, a wide range of stochastic processes with two-sided jumps can be realized.

The rest of the paper is organized as follows. In Section \ref{Sec:Prelim}, we introduce the relevant mathematical aspects of the driving stochastic process used in the models. We also introduce the general setting of the capital structure model for a firm, and the equity maximization problem. In Sections \ref{Sec:Opt:Cont} and \ref{Sec:Opt:Per}, we derive the fluctuation identities for spectrally positive \red{\lev} processes under continuous and periodic observations, respectively, and show the existence of the optimal bankruptcy barriers that solve the Leland-Toft problem both in the continuous and periodic cases. 
Section \ref{Sec:TwoStage} considers the two-stage problem to obtain the optimal capital structure, where we aim to maximize 
the value of the firm in terms of the total face value of the debt. 
Section \ref{Sec:Num} 
presents numerical examples that illustrate the results derived in the earlier sections. We conclude the paper in Section \ref{Sec:Final}. Some proofs and a review of the fluctuation identities for the spectrally negative \lev process are given in the Appendix.

\section{Preliminaries}\label{Sec:Prelim}
Let $(\Omega,\mathcal{F},\Prob)$ be a complete probability space hosting a \lev process $X = (X_t)_{t\geq 0}$ with $X_0=0$. To be more precise, $X = (X_t)_{t\geq 0}$ is a real-valued stochastic process with independent and stationary increments, and with càdlàg sample paths such that $X_0 = 0$. 
In this work, we will work with \textit{spectrally positive} \lev processes, i.e., L\'evy processes without negative jumps. These processes are characterized by their \textit{Laplace exponent} $\psi: [0,\infty) \rightarrow \R$
\[
\E \left[ e^{-\theta X_t} \right] = e^{t\psi(\theta)}, \quad t,\theta \geq 0,
\]
where, by the \textit{L\'evy-Khintchine formula},
\begin{equation}\label{Laplace:Exponent}
\psi(\theta) := \gamma \theta + \frac{\sigma^2}{2}\theta^2 + \int_{(0,\infty)}\left( e^{-\theta z} - 1 + \theta z \Ind_{\lbrace z < 1 \rbrace} \right) \Pi(dz),
\end{equation}
where $\gamma\in\R, \, \sigma \geq 0$, and $\Pi$ is a measure, called the \textit{L\'evy measure}, supported on $(0,\infty)$ that satisfies $\int_{(0,\infty)} (1 \wedge z^2) \Pi(dz) < \infty.$

It is known that $X$ has paths of bounded variation if and only if $\sigma=0$ and $\int_{(0,1)} z \Pi(d z) < \infty$; in this case we can write
\[
X_t = ct + S_t, \quad t\geq 0,
\]
where $c := - \big(\gamma + \int_{(0,1)} z  \Pi(d z) \big)$ and $S = (S_t)_{t\geq 0}$ is a driftless subordinator. In order to avoid having processes with monotonous paths, we set $c < 0$ (when $X$ is of bounded variation). 
\subsection{Formulation of the Leland-Toft model} 
The value of the firm's asset is assumed to evolve according to an exponential L\'evy process given by
\[
V_t := V e^{X_t}, \quad t\geq 0,
\]
where the initial value $V$ is strictly positive. Let $r > 0$ be the risk-free interest rate and $0\leq  \delta < r$ the total payout rate to the firm's investors. 

The firm is partly financed by debt, which is being constantly retired and reissued in the following way: for some given constants $p,m >0$, the firm issues new debt at a constant rate with maturity profile $\phi(s):= m e^{-ms}$. Consequently, the face value of the debt issued in the time interval $(t,t+dt)$ that matures in the time interval $(t+s,t+s+ds)$ is given by $p\phi(s)dtds$. Assuming an infinite past, the face value of the debt held at time 0 and maturing in the interval $(s,s+ds)$ is
\[
\left[ \int_{-\infty}^0 p \phi(s-u) du \right] ds = p e^{-ms}ds,
\]
and hence the \textit{total value of the debt} is a constant given by
\[
P := p \int_0^{\infty} e^{-ms} ds = \frac{p}{m}.
\]

A bankruptcy event is determined \textit{endogenously} by the firm's shareholders, in terms of the firm's assets and a given set of times $\T \subseteq [0,\infty)$. Given an asset value level $V_B\geq 0$, bankruptcy is triggered at the first time the asset value is \textit{observed} to fall below this threshold. In other words, we define the time of bankruptcy as
\[
\sigma_{V_B}^- := \inf \lbrace t\in \T : V_t < V_B \rbrace.
\]
Here and for the rest of the paper, we set $\inf \emptyset = \infty$. We will study the case of \textit{continuous-observations}, i.e., when $\T = [0,\infty)$; as well as the case of \textit{periodic observations}, where $\T$ is the set of jump times of a time-homogeneous Poisson process with rate $\lambda > 0$ independent of $X$.

Upon bankruptcy, a fraction $\eta \in (0,1)$ of the asset value is lost due to re-organization and operational costs, while the remaining asset value is paid out to bondholders. We assume that the debt is of equal seniority, so that a fraction $\frac{1}{P}$ of the remaining asset value is distributed equally among bondholders.

The debt pays out coupons at a constant rate $\rho > 0$, which are accumulated until maturity or bankruptcy, whichever comes first. Note that the remaining coupon payments are lost when bankruptcy occurs. In this setting, the value of the debt with unit face value and maturity date $t>0$ is
\begin{equation}\label{LT:Debt}
d(V;V_B,t) := \E\left[ \int_0^{t\wedge \sigma_{V_B}^-} \rho e^{-r s}ds \right] + \E\left[ e^{-rt} ; t < \sigma_{V_B}^- \right] + \frac{1}{P} \E\left[ (1-\eta) e^{- r \sigma_{V_B}^-} V_{\sigma_{V_B}^-} ; \sigma_{V_B}^- \leq t \right].
\end{equation}
By integrating expression \eqref{LT:Debt} and using Fubini's Theorem, the \textit{total value of the debt} becomes
\begin{align}
\D(V;V_B) &:= \int_0^{\infty} p e^{-mt} d(V;V_B,t) dt \nonumber\\
 &= \E\left[ \int_0^{\sigma_{V_B}^-} e^{-(r+m)t}P(\rho + m) dt \right] + (1-\eta) \E\left[ e^{-(r+m) \sigma_{V_B}^-} V_{\sigma_{V_B}^-} ; \sigma_{V_B}^-<\infty \right].\label{Total_Debt}
\end{align}
\noindent We follow the \textit{trade-off} theory and we assume that the \textit{value of the firm} is determined by tax rebates on coupon payments, and by bankruptcy costs. Regarding the tax rebates, we assume that there is \green{a} corporate tax rate $\kappa > 0$ and a cutoff level $V_T \geq 0$ that determines the effect of tax rebates on coupons. Whenever the asset value is above this level, the tax rebates are accrued continuously at rate 
$\kappa \rho P$, otherwise these rebates are equal to 0. On the other hand, we know that bankruptcy costs are given by $\eta V_{\sigma_{V_B}^-}$. Under these assumptions the value of the firm is
\begin{equation}\label{Market_Value_Firm}
\V(V;V_B) := V + \kappa \rho P \E\left[ \int_0^{\sigma_{V_B}^-} e^{-rt} \Ind_{\lbrace V_t \geq V_T \rbrace} dt \right] - \eta \E\left[ e^{-r \sigma_{V_B}^-} V_{\sigma_{V_B}^-} ; \sigma_{V_B}^- < \infty \right].
\end{equation}
Finally, the value of the firm's equity is 
\begin{equation}\label{Defn_Equity_Value}
\Eq(V;V_B) = \V(V;V_B) - \D(V;V_B).
\end{equation}
The problem is to find the \textit{optimal bankruptcy threshold} $V_B^* \geq 0$ that maximizes the equity value subject to the \textit{limited liability constraint}
\begin{align}
\Eq(V;V_B) \geq 0, \quad V \geq V_B. \label{limited_liability}
\end{align}
\section{Optimal bankruptcy level under continuous observations}\label{Sec:Opt:Cont}
Our first subject of study will be solving the optimal capital structure problem under the assumption that the firm's assets are observed continuously. In this model the bankruptcy decision is taken instantaneously, as soon as the asset value falls below the level $V_B$, 
and hence the bankruptcy time is defined as
\begin{align*} 
\sigma_{V_B}^- := \inf\lbrace t\geq 0 : V_t < V_B \rbrace. 
\end{align*}
We derive expressions for the debt, firm, and equity values 
by using results from the \textit{Fluctuation Theory} for spectrally one-sided L\'evy processes. From these expressions, we will be able to analyze and solve the Leland-Toft optimal capital structure problem, i.e., finding the bankruptcy-triggering level that maximizes the equity value under the limited liability constraint \eqref{limited_liability}.
\subsection{Fluctuation identities for spectrally positive \lev processes}
We present some fluctuation identities for spectrally positive  \lev  processes, 
which will be given in terms of the so-called scale functions. 
For $x\in\R$, we denote the expectation operator associated with the spectrally positive L\'evy process $X$ started at $x$ by $\E_x$.

Recall the Laplace exponent of the spectrally positive  \lev  process given in \eqref{Laplace:Exponent}. We define its right-continuous inverse 
by:
\[ \Phi(q) := \sup \lbrace p>0: \psi(p) = q \rbrace, \quad q \geq 0. \]
\begin{mydef}[Scale function] For each $q\geq0$, the $q$-scale function $W^{(q)}:\R \rightarrow [0,\infty)$ is the unique function such that $W^{(q)}(x)=0$ for all $x<0$, and, on $[0,\infty)$, it is a strictly increasing and continuous function whose Laplace transform is given by
	\[ \int_0^{\infty} e^{-\lambda x} W^{(q)}(x) dx = \frac{1}{\psi(\lambda) - q}, \quad \lambda > \Phi(q). \]
	We also define the second scale function:
	\begin{equation} \label{2nd:Scale:Fn}
	Z^{(q)}(x;\theta) := e^{\theta x} \left( 1 + (q - \psi(\theta)) \int_0^x e^{-\theta z} W^{(q)}(z) dz \right), \quad x\in\R, \; \theta \geq 0.
	\end{equation}
	In particular, we write $Z^{(q)}(x) = Z^{(q)}(x;0)$.
\end{mydef}

\begin{remark}[Smoothness of the scale function] \label{remark_smoothness_scale_function} For the case where $X$ is of unbounded variation, $W^{(q)}(0) = 0$. On the other hand, when $X$ is of bounded variation, $W^{(q)}(0) > 0$. For both cases, the right-hand derivative of $W^{(q)}$ exists for all $x \in (0,\infty)$. See, Chapter 8 of \cite{Kyprianou2007} for more details and related results.
\end{remark}

For the spectrally positive \lev process $X$ we denote the first passage time below and above $0$ by
\begin{equation}\label{fptc} \tau_0^{-} := \inf \lbrace t>0: X_t <0 \rbrace, \qquad \tau_0^{+} := \inf \lbrace t>0: X_t >0 \rbrace.
\end{equation}
Due to the absence of negative jumps, we have $\mathbb{P}_x$-a.s.\ for $x \geq 0$,
\begin{equation}\label{creeping} X_{\tau_0^{-}} = 0. \end{equation}

\begin{lemma}\label{Lem:1st:FI} For all $x\geq0$ and $q \geq 0$, the Laplace transform of $\tau_0^{-}$ is given by 
	\begin{equation}\label{1st:Fluc:Id}
	\gamma(x;q):=\E_{x} \left[1 - e^{-q\tau_0^{-}} ; \tau_0^{-} < \infty \right] = 1-e^{- \Phi(q) x}.
	\end{equation}
\end{lemma}

\noindent Note that $\gamma(\cdot;q)$ is differentiable on $(0,\infty)$, with its derivative given by 
\begin{equation}\label{der_gamma} \frac{\partial \gamma}{\partial x}(x;q) = \Phi(q) e ^{- \Phi(q) x}, \  x>0. 
\end{equation}
Moreover, $\gamma$ is right-differentiable at $0$, and its right-hand derivative is given by
\begin{align}
\frac{\partial \gamma}{\partial x}(0+;q) = \Phi(q).  \label{gamma_der_zero}
\end{align}

\begin{remark}
	Using \eqref{creeping} we have that for any $\beta\in\R$  and $x \geq 0$,
	\begin{align*}
	\E_{x} \left[1 - e^{-q\tau_0^{-} + \beta X_{\tau_0^{-}}} ; \tau_0^{-} < \infty \right]=\E_{x} \left[ 1 - e^{-q\tau_0^{-}} ; \tau_0^{-} < \infty \right]=\gamma(x;q).
	\end{align*}
\end{remark}

\noindent We also define the function $g(x;q,a)$ as 
\[ g(x;q,a) := \E_{x}\left[ \int_0^{\tau_0^-} e^{-qt} \Ind_{ \lbrace X_t \geq a \rbrace } dt \right], \quad x \geq 0,\, q\geq0,\, a\in\R. \]
Observe that $g(x;q,a)=0$ for all $x < 0$ since $\lbrace X_0 = x \rbrace$ implies that $\tau_0^-=0$ almost surely.
The following result gives an explicit expression for the function $g$.
\begin{lemma}\label{2nd:Fluc:Id_lemma} For $x \geq0,\,q\geq 0,$ and $a \in\R$ we have 
	\begin{equation*} 
	g(x;q, a) = \frac{1}{q}\left( 1 - e^{-\Phi(q) x} \right) + \overline{W}^{(q)}(a-x) - e^{-\Phi(q)x} \overline{W}^{(q)}( a),
	\end{equation*}
	where
	\[ \overline{W}^{(q)}( x) := \int_0^{ x}W^{(q)}(u)du, \quad x \in\R. \]
\end{lemma}

\noindent Note that this function $g(\cdot;q, a)$ is also differentiable on $(0,\infty) \backslash \{a\}$ 
\begin{equation}\label{der_g} 
 \frac{\partial g}{\partial x}(x;q, a) = e^{-\Phi(q) x} \frac{\Phi(q)}{q} +  \Phi(q)e^{-\Phi(q) x} \overline{W}^{(q)}(a) - W^{(q)}(a- x). 
\end{equation}
It is also right-differentiable at $0$ with right-hand derivative given by
\begin{align} 
\frac{\partial g}{\partial  x}(0+;q, a) = \frac{\Phi(q)}{q}  +  \left(\Phi(q) \overline{W}^{(q)}(a) -  W^{(q)}(a)\right)\Ind_{\{a > 0\}}. \label{g_derivative_zero}
\end{align}
\subsection{Determining the optimal bankruptcy level}\label{SubSec:Opt:Cont}
Recall that in the context of the Leland-Toft model, 
we are seeking to maximize the expected equity value upon the event of bankruptcy. 

Suppose $V_B > 0$ and $V_T > 0$.
When $V \leq V_B,$ 
we have that $\sigma_{V_B}^- = 0$, and hence $\D(V;V_B) = \V(V;V_B) = (1 -\eta) V$ implying $\mathcal{E}(V;V_B) = 0$. Hence, 
throughout the rest of the section we assume $V \geq V_B$.
Based on the calculations in Lemmas \ref{Lem:1st:FI} and \ref{2nd:Fluc:Id_lemma}, and using the fact that $V_{\sigma_{V_B}^-}=V_B$ on $\{\sigma_{V_B}^- < \infty \}$ since $(V_t)_{t \geq 0}$ has no negative jumps, we can express the value of the firm and the total value of the debt as 
\begin{align}
\mathcal{V}(V;V_B) &= V + \kappa \rho P g \left( \log \frac{V}{V_B} ;r, \log \frac{V_T}{V_B} \right) - \eta V_B \left[ 1-\gamma \left( \log \frac{V}{V_B} ;r \right) \right], \label{Mkt:Val}\\
\mathcal{D}(V;V_B) &= \frac{(\rho + m)P}{r+m} \gamma \left( \log \frac{V}{V_B} ;r+m \right) +  (1-\eta)V_B \left[1 - \gamma \left( \log \frac{V}{V_B} ;r+m \right)\right].\label{Debt:Val}
\end{align}
 Hence, the equity value is given by 
\begin{equation}\label{Eq:Val}
\begin{split}
\Eq(V; V_B) &= V - V_B + V_B \left(\eta \gamma \left( \log \frac{V}{V_B} ;r \right) + (1-\eta)\gamma \left( \log \frac{V}{V_B} ;r+m \right) \right) \\
& \qquad + \left[ \kappa \rho  g \left( \log \frac{V}{V_B} ;r, \log \frac{V_T}{V_B} \right) - \frac{(\rho + m)}{r+m} \gamma \left( \log \frac{V}{V_B} ;r+m \right) \right] P.
\end{split}
\end{equation}
In particular, when $V_T=0$ we have $g\left( \log \frac{V}{V_B} ;r, \log \frac{V_T}{V_B}\right) = \frac{1}{r} \gamma\left( \log \frac{V}{V_B} ;r\right)$, and the equity value is given by 
\begin{equation}\label{Eq:Val:VT_0}
\begin{split}
\Eq(V;V_B) &= V - V_B + V_B \left( \eta \gamma\left( \log \frac{V}{V_B} ;r \right) + (1-\eta)\gamma\left( \log \frac{V}{V_B} ;r+m \right)\right) \\
&\qquad + \left[ \frac{\kappa \rho }{r} \gamma\left( \log \frac{V}{V_B} ;r \right) - \frac{(\rho + m)}{r+m} \gamma\left( \log \frac{V}{V_B} ;r+m\right) \right] P.
\end{split}
\end{equation} 
For the case $V_B = 0$, the equity value is given later in \eqref{Eq:Val:VB_0} (as we see immediately below, when the optimal barrier is zero, necessarily $V_T = 0$).

\begin{remark} \label{remark_E_V_B_zero} Since 
	 $\gamma$ and $g$ satisfy 
$\gamma(0+;q) =0$ and $g(0+;q,a) =0$, 
we have that $\Eq(V_B+; V_B) =0$ for all $V_B > 0$.
\end{remark}
\noindent Our aim is to show that $V_B^*$ satisfying 
\begin{equation}\label{Smt:Pst:Cont}
V_B^* = \inf \left\{ V_B > 0:  \frac{\partial \Eq}{\partial V}(V_B+; V_B) > 0 \right\},
\end{equation}
is optimal if such level exists. 
First, by \eqref{Eq:Val} and \eqref{Eq:Val:VT_0}, together with \eqref{gamma_der_zero} and \eqref{g_derivative_zero},
\begin{equation}\label{Aux:Deriv}
\frac{\partial \Eq}{\partial V}(V_B+; V_B) = \frac{
C
}{V_B} f(V_B),
\end{equation}
where 
\begin{align}
C &:= 1 + \eta\Phi(r) + (1-\eta) \Phi(r+m),  \nonumber\\	
\label{Aux:FP}
f(v) &:= v - \left(  K + \Ind_{\{V_T > 0\}} \kappa\rho \big(W^{(r)}(\log(V_T/v)) - \Phi(r)\overline{W}^{(r)}(\log(V_T/v)) \big)\right) \frac{P}{C}, \quad v > 0,
\end{align}
with
\begin{equation}\label{LT:Cont:K} 
K := \frac{\rho +m}{r+m}\Phi(r+m) - \frac{\kappa\rho }{r}\Phi(r).
\end{equation}
Using that $\eta \in [0,1]$ and $\Phi(r), \Phi(r+m) > 0$ we note that $C > 0$.


In other words, condition \eqref{Smt:Pst:Cont} is equivalent to
\begin{equation}\label{Cont:Equiv:FP}
V_B^*=\inf\{V_B>0:f(V_B) >0\}.
\end{equation}
\begin{remark}\label{Remark_f}
By Remark \ref{remark_smoothness_scale_function} we have that the function $f$ given in \eqref{Aux:FP}, is continuous on $(0,\infty)$ for the case $V_T=0$, or when the process $X$ is of unbounded variation. In these cases, condition \eqref{Cont:Equiv:FP} is equivalent to finding $V_B^*$ such that $f(V_B^*)=0$ or equivalently $\frac{\partial \Eq}{\partial V}(V_B^*+; V_B^*)=0$ (smooth fit). In general, when $V_B^* \neq V_T$, $V_B^*$  satisfies the smooth fit condition. 
\end{remark}


We will find conditions under which the level satisfying \eqref{Cont:Equiv:FP} exists and is unique. 
\begin{lemma}\label{lemma_f} The function $f$ is strictly increasing on $(0,\infty)$ with
\begin{align*}
\lim_{v\downarrow 0}f(v)&=
\begin{cases} - \frac{\rho +m}{r+m}\frac{\Phi(r+m)}{C} P &\mbox{if } V_T>0, \\
 -\frac{ K}{C}P & \mbox{if } V_T=0,
\end{cases}\\
\lim_{v\uparrow\infty} f(v) &= \infty.
\end{align*}
\end{lemma}
\begin{proof}
(i)
First, for $V_T > 0$, by differentiating \eqref{Aux:FP},
\[ f'(v) = 1+ \left(\frac{\kappa\rho P}{v}\right) \left( \frac{A(v)}{
C
} \right), \quad v \in (0,\infty) \backslash \{ V_T \}, \]
with
\begin{equation*} 
A(v) := W^{(r)'}(\log(V_T/v)) - \Phi(r) W^{(r)}(\log(V_T/v)),
\end{equation*}
where the term $W^{(r)\prime}$ is understood as the right-hand derivative of $W^{(r)}$, if it is not differentiable (see Remark \ref{remark_smoothness_scale_function}). 
The function $A(v)$ is non-negative because it is the $r$-resolvent density of the descending ladder height process of $-X$ 
(see Section 2.4 in Pistorius \cite{Pistorius2007}). 
It follows that
\begin{equation}\label{increasing_f_1} f'(v) > 0, \quad v \in (0,\infty) \backslash \{ V_T \}, \end{equation}
On the other hand, using \eqref{Aux:FP} we note that 
	\begin{equation}\label{increasing_f_2}
	f(V_T-)=V_T- \left(  K + 
	\kappa\rho W^{(r)}(0)
	 \right) \frac{P}{C} <V_T-\frac{KP}{
	C
	}=f(V_T+). 
	\end{equation}
Hence, identities \eqref{increasing_f_1} and \eqref{increasing_f_2} imply that $f$ is strictly increasing on $(0,\infty)$.

Following equation (42) in \cite{Kuznetsov2013}, we have 
\begin{align*}
\E_{\log(v/V_T)} \left[ e^{-r \tau_0^{+}} ; \tau_0^{+} < \infty \right] &= Z^{(r)}(\log(V_T/v)) - \frac{r}{\Phi(r)}W^{(r)}(\log(V_T/v)), 
\end{align*}
and hence  we have
\begin{align}
W^{(r)}(\log(V_T/v)) - \Phi(r)\overline{W}^{(r)}(\log(V_T/v)) &= \frac{\Phi(r)}{r} \left( 1 - \E_{\log(v/V_T)} \left[ e^{-r \tau_0^{+}} ; \tau_0^{+} < \infty \right] \right).\label{FI:1st:PT:Cont}
\end{align}
This allows us to simplify the expression for $f$ as 
\begin{equation}\label{FP:Alt}
f(v) = v- \left( \frac{\rho +m}{r+m}\Phi(r+m) - \kappa\rho \frac{\Phi(r)}{r} \E_{\log(v/V_T)} \left[ e^{-r \tau_0^{+}} ; \tau_0^{+} < \infty \right]\right) \frac{P}{C}.
\end{equation}

As $v \downarrow 0$, we have $\log(v/V_T) \downarrow -\infty$. On the other hand, we have that the mapping $z \mapsto \E_{z} \left[ e^{-r \tau_0^{+}} ; \tau_0^{+} < \infty \right]$ is non-negative, and monotone increasing  on $(-\infty,0)$, hence the limit as $z \downarrow -\infty$ exists, and it is equal to 0. 
Thus, by substituting 
this limit into the alternate form of $f$ given in \eqref{FP:Alt}, we get
\begin{align*} 
\lim_{v\downarrow 0}f(v) &= - \Phi(r+m)\left( \frac{\rho +m}{r+m}
\right) \frac{P}{C}.
\end{align*}
On the other hand, by noting that $ \log(V_T/v) \downarrow -\infty$ as $v \uparrow \infty$ and the fact that
\[ W^{(r)}(\log(V_T/v)) - \Phi(r)\overline{W}^{(r)}(\log(V_T/v)) =0,\qquad\text{for $V_T/v < 1$,} \]
we have
\[
\lim _{v\uparrow \infty} f(v) = \infty.
\]

(ii) Now, when $V_T = 0$  we have that $f$ is linear with unit slope on $(0,\infty)$, hence it is strictly increasing and $\lim _{v\uparrow \infty} f(v) = \infty$. 
Additionally, by \eqref{Aux:FP} we have 
\[
\lim_{v\downarrow 0}f(v)   =- \left(\frac{ KP}{
C
}\right).
\]
\end{proof}
In view of  Lemma \ref{lemma_f}, we propose our candidate for $V_B^*$ as follows.
\begin{enumerate}
	\item For the case $V_T>0$ or $V_T=0$ with $K> 0$ 
	we set $V_B^*>0$ such that $V_B^*=\inf\{V_B>0:f(V_B) >0\}$, 
	which exists by Lemma \ref{lemma_f}.
	\item For the case $V_T=0$ with $K \leq 0$  
	we set $V_B^*=0$.
\end{enumerate}
The equity value for the case $V_B^*>0$ can be computed using identity \eqref{Eq:Val:VT_0}. For the case 
where $V_T=0$ with $K \leq 0$, 
the equity value is given by 
	\begin{equation}\label{Eq:Val:VB_0}
	\Eq(V;0) = V + \Big( \frac{\kappa \rho }{r} - \frac{\rho + m}{r+m} \Big) P.
	\end{equation}
\subsection{Optimality of $V_B^*$}

In order to show that $V_B^*$ is the optimal barrier, it suffices to verify the following.
\begin{enumerate}
\item Any level $V_B$ below $V_B^*$ violates the limited liability condition  \eqref{limited_liability}.
\item $V_B^*$ achieves a higher equity value than any level $V_B > V_B^*$ does.
\item $V_B^*$ fulfills the limited liability condition  \eqref{limited_liability}.
\end{enumerate}
These are confirmed in the following three propositions.
\begin{prop}\label{Prop:NLLC:Cont} Suppose $V_B^* > 0$. For $V_B < V_B^*$, the limited liability constraint \eqref{limited_liability} is not satisfied.
\end{prop}
\begin{proof}
Let $V_B < V_B^*$. 
Combining \eqref{Aux:Deriv} and \eqref{Aux:FP} with the fact that $f$ is strictly increasing as shown in Lemma \ref{lemma_f}, and the fact that $C > 0$, we get
\[ \frac{\partial \Eq}{\partial V}(V_B+; V_B) < 0, \]
which, combined with the fact that $\Eq(V_B+; V_B) =0$ as in Remark  \ref{remark_E_V_B_zero} and the fact that $\Eq(\cdot;V_B)$ is continuous, leads to the conclusion that there exists $V > V_B$ such that $\Eq(V;V_B) <0$. 
\end{proof}
\begin{prop}\label{Prop:Optim:Cont} 
	For all $V > V_B > V_B^*$, we have $\mathcal{E}(V; V_B^*) \geq  \mathcal{E}(V; V_B)$. 
\end{prop}
\begin{proof}
Fix $V > V_B^*$. 
We will show that the mapping $V_B \mapsto \Eq(V;V_B)$ is monotonically decreasing on $[V_B^*,V]$.
By differentiating with respect to $V_B$, we obtain
\begin{equation}\label{Opt:Aux1}
\frac{\partial \gamma}{\partial V_B}\left( \log \frac{V}{V_B} ;q\right) = -\left(\frac{\Phi(q)}{V_B}\right) \left(\frac{V_B}{V}\right)^{\Phi(q)}, 
\end{equation}
and, for $V_B\not=V_T$, 
\begin{equation}\label{Opt:Aux2}
\frac{\partial g}{\partial V_B}\left( \log \frac{V}{V_B} ;q,\log \frac{V_T}{V_B} \right) = -\left( \frac{1}{V_B} \right) \left[ \frac{\Phi(q)}{q} -  \Ind_{\lbrace V_T > V_B \rbrace} \left( W^{(q)}\left( \log \frac{V_T}{V_B} \right) - \Phi(q)\overline{W}^{( q )}\left( \log \frac{V_T}{V_B}\right)  \right) \right]\left(\frac{V_B}{V}\right)^{\Phi(q)}. 
\end{equation}

 By applying expressions \eqref{Opt:Aux1} and \eqref{Opt:Aux2} in \eqref{Eq:Val}, we get for $V_B\not=V_T$
\begin{align}\label{derequity_v_b}
\frac{\partial \Eq}{\partial V_B} (V; V_B) &= -1 + \eta \gamma\left( \log \frac{V}{V_B} ;r\right) + (1-\eta) \gamma\left( \log \frac{V}{V_B} ;r+m\right)\notag\\&+ \left( \frac{\Phi(r+m)}{V_B} \right) \left( \frac{(\rho +m)P}{r+m} \right) \left(\frac{V_B}{V}\right)^{\Phi(r+m)}\notag\\
 &- \left( \eta \Phi(r) \left(\frac{V_B}{V}\right)^{\Phi(r)} + (1-\eta) \Phi(r+m)  \left(\frac{V_B}{V}\right)^{\Phi(r+m)} \right) -  \left( \frac{\kappa \rho P}{V_B} \right) B\left(V_B,V_T\right)
\left(\frac{V_B}{V}\right)^{\Phi(r)}, 
\end{align}
where
\[
B\left(V_B,V_T\right) := \frac{\Phi(r)}{r} - \Ind_{\{V_T>V_B\}}\left[W^{(r)}\left(\log \frac{V_T}{V_B} \right) - \Phi(r)\overline{W}^{(r)}\left( \log \frac{V_T}{V_B} \right) \right].
\]
After rearranging terms, we can rewrite $\frac{\partial \Eq}{\partial V_B} (V; V_B)$ as
\begin{align}\label{Opt:Aux3}
\frac{\partial \Eq}{\partial V_B} (V;V_B) &= 
- \eta (\Phi(r) + 1)\left[\left(\frac{V_B}{V}\right)^{\Phi(r)} - \left(\frac{V_B}{V}\right)^{\Phi(r+m)} \right] \notag\\
  &- \left(\frac{V_B}{V}\right)^{\Phi(r+m)} \left( \frac{ 
  C
  }{V_B} \right) f(V_B) - \left(\frac{\kappa \rho P}{V_B} \right) B\left(V_B,V_T\right)
\left[ \left(\frac{V_B}{V}\right)^{\Phi(r)}- \left(\frac{V_B}{V}\right)^{\Phi(r+m)} \right],\quad V_B\not=V_T.
\end{align}
Note that the first and second terms in the right-hand side of \eqref{Opt:Aux3} are non-positive since $V\geq V_B$, $C > 0$,
$\Phi$ is a non-negative and strictly increasing 
function, and because $f(V_B) \geq 0$ for $V_B \in (V_B^*, V)$. For the last term, if $V_B\geq V_T$, $B\left(V_B,V_T\right) = \Phi(r) /r > 0$ and otherwise we use the fluctuation identity \eqref{FI:1st:PT:Cont} 
 to write
\[
B\left(V_B,V_T\right)
= \frac{\Phi(r)}{r} \E_{- \log \frac{V_T}{V_B}}  \left[ e^{-r \tau_0^+} ; \tau_0^+ < \infty \right] > 0. \]
Hence the last term in \eqref{Opt:Aux3} is non-positive as well.
Thus, for all 
$V_B \in (V_B^*, V) \backslash\{V_T\}$, 
\begin{align}
\frac{\partial \Eq}{\partial V_B} (V;V_B) < 0. \label{E_V_B_der_negative}
\end{align}
Hence, the proof follows from the continuity of the mapping $V_B\mapsto \Eq(V;V_B)$ together with \eqref{E_V_B_der_negative}.
\end{proof}
\begin{prop}\label{Prop:LLC:Cont}The level $V_B^*$ satisfies the limited liability constraint  \eqref{limited_liability}.
\end{prop}
\begin{proof}

(i) Suppose $V_B^* > 0$. For $V \geq V_B^*,$ by differentiating the equity value with respect to $V$ and using \eqref{der_gamma} together with \eqref{der_g} we get for $V\not=V_T$ 
\begin{align*}
\frac{\partial \Eq}{\partial V}(V;V_B^*) &= 1 + \frac{V_B^*}{V} \left[ \eta \Phi(r) \left(\frac{V_B^*}{V}\right)^{\Phi(r)} + (1-\eta) \Phi(r+m) \left(\frac{V_B^*}{V}\right)^{\Phi(r+m)} \right] \\
&+ \frac{P\kappa\rho}{V}  \left[\frac{\Phi(r)}{r} \left(\frac{V_B^*}{V}\right)^{\Phi(r)} +  \left( \Phi(r)\left(\frac{V_B^*}{V}\right)^{\Phi(r)} \overline{W}^{(r)}\left( \log\frac{V_T}{V_B^*}\right) \Ind_{\lbrace V_T > V_B^* \rbrace} - W^{(r)}\left( \log \frac{V_T}{V} \right) \Ind_{\lbrace V_T > V\rbrace}\right) \right] \\
&- \frac{1}{V} \left(\frac{P(\rho+m)}{r+m}\right)\Phi(r+m) \left(\frac{V_B^*}{V}\right)^{\Phi(r+m)}.
\end{align*}
 After rearranging terms and using \eqref{derequity_v_b}, we obtain for $V\geq V_B^*$ and $V\not=V_T$
\begin{align}\label{aux_1}
\frac{\partial \Eq}{\partial V}(V;V_B^*) &= 1 - \left( \frac{V_B^*}{V} \right) \left[ \eta \left(\frac{V_B^*}{V}\right)^{\Phi(r)} + (1-\eta) \left(\frac{V_B^*}{V}\right)^{\Phi(r+m)} \right] -  \left( \frac{V_B^*}{V} \right) \frac{\partial \Eq}{\partial V_B}(V;V_B^*)\notag\\
& \quad +  \frac{P\kappa\rho}{V}\left[ \Ind_{\{V_T >V_B^*\}}W^{(r)}\left( \log \frac{V_T}{V_B^*} \right) \left(\frac{V_B^*}{V}\right)^{\Phi(r)} - \Ind_{\{V_T >V\}}W^{(r)}\left( \log \frac{V_T}{V} \right)  \right],
\end{align}
where we understand $\frac{\partial \Eq}{\partial V_B}(V;V_B^*)$ as the right-hand side of \eqref{derequity_v_b} for the case $V_B^*=V_T$.

The right hand-side of \eqref{aux_1} is positive.
Indeed, the sum of the first two terms in the right hand side of \eqref{aux_1} is non-negative, since $V \geq V_B^*$, $\eta \in [0,1]$, and $\Phi$ is strictly increasing; the third term is non-negative in light of Proposition \ref{Prop:Optim:Cont} by  \eqref{E_V_B_der_negative}; 
using that $V\geq V_B^*$ together with  \eqref{Thm:Res:Dens:Cont} (in the appendix) we obtain that the fourth term is non-negative as well.
Hence using \eqref{aux_1} together with the fact that the mapping $V\mapsto\Eq(V;V_B^*)$ is continuous and Remark \ref{remark_E_V_B_zero} we obtain that  $\Eq(V;V_B^*)\geq 0$ for $V\geq V_B^*$. 

(ii) 
Finally, for $V_B^* = 0$, where necessarily $V_T=0$ and $K \leq 0$ and hence
	\begin{align*}   \frac{\rho +m}{r+m}- \frac{\kappa\rho }{r} \leq \frac{\rho +m}{r+m}\frac {\Phi(r+m) } {\Phi(r)}- \frac{\kappa\rho }{r}=\frac{K}{\Phi(r)}\leq 0,
	\end{align*}
	it follows from identity \eqref{Eq:Val:VB_0} that $\mathcal{E}(V;0) \geq 0$.

\end{proof}
\noindent Finally, by Propositions \ref{Prop:NLLC:Cont}, \ref{Prop:Optim:Cont} and \ref{Prop:LLC:Cont}, we state our main result.
\begin{theo} \label{theorem_main_classical}
The optimal bankruptcy level is given by $V_B^*$.
\end{theo}
\section{Optimal bankruptcy level under periodic observations}\label{Sec:Opt:Per}
\noindent We turn our attention to the problem of determining the optimal bankruptcy-triggering level when the asset value process, $(V_t)_{t\geq 0}$, is not observed continuously, but at a set of discrete times given by the arrival times of an independent Poisson process with intensity $\lambda > 0$. We denote the collection of these times by $\mathcal{T} =(T_{i}^{\lambda})_{ i \geq 1}$. This assumption makes the model more realistic in the sense that shareholders do not observe the evolution of the firm's asset 
continuously and instead they observe it 
at discrete times, and hence the decision to declare bankruptcy is taken at these observation dates. 

For such a collection of times and a bankruptcy barrier $V_B \geq 0$, we define the \textit{first observed bankruptcy time} as the stopping time
\begin{equation}\label{Bkrpt_Time_Periodic}
\sigma_{V_B}^{-,\lambda} := \inf\lbrace  T^{\lambda} \in \mathcal{T}  : V_{T^{\lambda}} < V_B \rbrace.
\end{equation}
It is more realistic to assume that  the bankruptcy decision can be made at time zero so that the bankruptcy time becomes
$\overline{\sigma}_{V_B}^{-,\lambda} := \inf\lbrace T^{\lambda} \in \mathcal{T}\cup \{0\}: V_{T^{\lambda}} < V_B \rbrace = \sigma_{V_B}^{-,\lambda} \Ind_{\lbrace V \geq V_B \rbrace}$. 
However, similar to the continuous observation case, 
we obtain zero equity value if the initial value $V < V_B$ and therefore we 
focus on the case $V \geq V_B$ and hence $ \overline{\sigma}_{V_B}^{-,\lambda} = \sigma_{V_B}^{-,\lambda}$ in the rest of this section.

In this setting, we are interested in proving the existence of  an \textit{optimal bankruptcy level}, denoted by $V_B^{*,\lambda}$, which maximizes the equity value \eqref{Defn_Equity_Value} and satisfies the \textit{limited liability constraint} \eqref{limited_liability}. 
\begin{remark} \label{remark_convergence} Intuitively we expect that the \textit{first observed bankruptcy time} will converge 
to the \textit{classical bankruptcy time} as $\lambda \rightarrow \infty$. We also expect that the expressions and results derived in this Section will converge to the ones shown in Subsection \ref{SubSec:Opt:Cont} as $\lambda \rightarrow \infty$. These convergence results are confirmed numerically in Section \ref{Sec:Num}.
\end{remark}
\noindent We will develop expressions for the first observed passage time \eqref{Bkrpt_Time_Periodic}, the deficit 
and the resolvent measure under this setting of discrete time observations; in turn, these expressions will allow us to express the total value of debt \eqref{Total_Debt}, the value of the firm \eqref{Market_Value_Firm},  and the equity value \eqref{Defn_Equity_Value} in terms of the corresponding fluctuation identities.
\subsection{Fluctuation identities for spectrally positive processes under periodic observations}\label{fluc_iden_periodic}
In this setting, we are interested in the first \textit{observed} time below $0$ of the \lev process $X$, which is defined as 
\begin{equation} \label{Sp_Pos_Poisson_Passage} 
T_0^{-} := \inf \lbrace T^{\lambda} \in \mathcal{T}: X_{T^{\lambda}} < 0 \rbrace,
\end{equation}
as well as in the value of 
$X_{T_0^{-}}$. 
In our setting, for a spectrally positive L\'evy process started at $\log (V/V_B)$ the bankruptcy time given in \eqref{Bkrpt_Time_Periodic} is equal 
to $T_0^{-}$; similarly, the asset value at bankruptcy $V_{\sigma_{V_B}^{-,\lambda}}$ is given by $V_B \exp (X_{T_0^-})$. The following lemma  gives the joint Laplace transform of $(T_0^-, X_{T_0^{-}})$. This is a direct consequence of  Theorem 3.1 in  \cite{albrecher2016}; the proof is deferred to Appendix \ref{proof_1st:FI:Per_lemma}. 
\begin{lemma}\label{1st:FI:Per_lemma} For $x,\beta \geq 0$ and $q >0$, the joint Laplace transform of $(T_0^-, X_{T_0^{-}})$ is given by
	\begin{equation}\label{1st:FI:Per}
	J^{(q,\lambda)}(x;\beta) := \E_x \left[ e^{-q T_0^- + \beta X_{T_0^{-}} } ; T_0^{-} < \infty \right] = \frac{\Phi(\lambda + q) - \Phi(q)}{\beta + \Phi(\lambda+ q)} e^{-\Phi(q) x}.
	\end{equation}
\end{lemma}
\noindent On the other hand, using the resolvent 
measure for spectrally negative \lev processes observed at Poisson arrival times derived by Landriault et al. \cite{LANDRIAULT2018152}, we can obtain an expression for
\begin{equation}\label{fun_Lambda}
\Lambda^{(q,\lambda)} (x,z) := \E_x \left[ \int_0^{T_z^-} e^{-qt} \Ind_{\lbrace X_t \geq \log V_T \rbrace} dt \right].
\end{equation}
The proof of the following proposition is given in Appendix \ref{appendix_2nd:FI:Per_lemma}.
\begin{prop}\label{2nd:FI:Per_lemma} 
	Let $q > 0$, and $x, z\in\R$ be fixed with $x \geq z$. 
	We have 
\begin{align}
\Lambda^{(q,\lambda)} (x,z) &= \frac{1}{q}\left( 1- J^{(q,\lambda)}(x-z;0) \right) - \frac{\Phi(\lambda + q) - \Phi(q)}{\Phi(\lambda +q)} e^{\Phi(q)(z-x)} \overline{W}^{(q)}(\log V_T -z)\Ind_{\lbrace V_T > 0 \rbrace} \nonumber \\
&- \frac{\Phi(\lambda + q) - \Phi(q)}{\lambda \Phi(\lambda +q)} e^{\Phi(q)(z- x)} Z^{(q)}(\log V_T - z;\Phi(\lambda + q))\Ind_{\lbrace V_T > 0 \rbrace} \nonumber \\
& + \overline{W}^{(q)}(\log V_T - x) \Ind_{\lbrace V_T > 0 \rbrace}.
\label{Res:Dens:Tax}
\end{align}		
\end{prop}
\begin{remark} \label{remark_Lambda}
	For $q > 0$ and $z \in \R$ we have
		\begin{align*}
	\Lambda^{(q,\lambda)}(z,z) &=   \frac{\Phi(q)}{\Phi(\lambda + q)}  \Big[ \frac{1}{q} + \Ind_{\lbrace V_T > 0 \rbrace}\overline{W}^{(q)}(\log V_T - z) \Big] \\ &-  \left( \frac{\Phi(\lambda + q) - \Phi(q)}{\lambda \Phi(\lambda + q)} \right)Z^{(q)}(\log V_T - z ; \Phi(\lambda + q))\Ind_{\lbrace V_T >0 \rbrace}. 
	\end{align*}
\end{remark}
\subsection{Expression for the firm/debt/equity values in terms of the scale functions}
We will use the expressions derived in Section \ref{fluc_iden_periodic} to write the firm, debt and equity values. 

Suppose $V_B > 0$. The value of the firm \eqref{Market_Value_Firm} for  $\sigma_{V_B}^- = \sigma_{V_B}^{-,\lambda}$ is given by 
\begin{align}
\mathcal{V}(V;V_B) &= V + \E\left[ \int_0^{\sigma_{V_B}^{-,\lambda}} e^{-rt}  \Ind_{\lbrace V_t \geq V_T \rbrace} P\kappa\rho dt \right] - \eta \E\left[ e^{-r\sigma_{V_B}^{-,\lambda}} V_{\sigma_{V_B}^{-,\lambda}} ; \sigma_{V_B}^{-,\lambda} < \infty \right] \nonumber \\
&= V + P\kappa\rho \Lambda^{(r,\lambda)}(\log V, \log V_B) - \eta V_B J^{(r,\lambda)}\left( \log \frac{V}{V_B};1 \right).\label{Val:Per}
\end{align}
On the other hand, 
the total value of the debt \eqref{Total_Debt} for  $\sigma_{V_B}^- = \sigma_{V_B}^{-,\lambda}$ is 
\begin{align}
\mathcal{D}(V;V_B) &= \E\left[ \int_0^{\sigma_{V_B}^{-,\lambda}} e^{-(r+m)t}(P\rho + p) dt \right] + (1-\eta)\E\left[ e^{-(r+m)\sigma_{V_B}^{-,\lambda}} V_{\sigma_{V_B}^{-,\lambda}} ; \sigma_{V_B}^{-,\lambda} < \infty \right] \nonumber \\
&= \frac{(\rho + m)P}{r+m}\left( 1- J^{(r+m,\lambda)}\left( \log \frac{V}{V_B} ;0 \right) \right) + (1-\eta) V_B J^{(r+m,\lambda)}\left( \log \frac{V}{V_B} ;1 \right).\label{Debt:Per}
\end{align}
By taking the difference between these, the equity value 
\eqref{Defn_Equity_Value} for  $\sigma_{V_B}^- = \sigma_{V_B}^{-,\lambda}$
is 
\begin{align}
\Eq(V;V_B) &= V + P\kappa\rho \Lambda^{(r,\lambda)}(\log V, \log V_B) - \eta V_B J^{(r,\lambda)}\left( \log \frac{V}{V_B};1 \right) \nonumber\\
 &\quad - \frac{(\rho + m)P}{r+m}\left( 1- J^{(r+m,\lambda)}\left( \log \frac{V}{V_B} ;0 \right) \right) - (1-\eta) V_B J^{(r+m,\lambda)}\left( \log \frac{V}{V_B} ;1 \right). \label{Eq:Per:Obs}
\end{align}
The case $V_B = 0$ is described later as we will see that when the optimal barrier is zero, we necessarily have
 $V_T = 0$.
\subsection{Determining the optimal bankruptcy level}\label{SubSec:Opt:Per}
Now we can move to finding the optimal bankruptcy threshold, denoted by $V_B^{*,\lambda}$. 
Our approach will be different from the calculations done in Section \ref{Sec:Opt:Cont}, since now we will find the optimal barrier through a \textit{continuous fit} approach. In other words, we will find the optimal barrier $V_B^{*,\lambda}$ as the solution of the equation
\begin{align}
\Eq(V_B^{*,\lambda}; V_B^{*,\lambda}) =0. \label{V_B_equality}
\end{align}
Our objective in this section is to show that there exists a unique solution to \eqref{V_B_equality} under  a certain condition, and that it is optimal in the sense of maximizing the equity value under the limited liability constraint \eqref{limited_liability}.

First, note that for $V_B > 0$
\begin{align}
\Eq(V_B;V_B) &= V_B C_\lambda
+ P\kappa \rho \Lambda^{(r,\lambda)}(\log V_B, \log V_B) - \frac{(\rho + m)P}{r +m }\frac{\Phi(r+m)}{\Phi(\lambda + r+m)}, \label{Eq:Per:Diag}
\end{align}
where
\begin{align} \label{C_lambda}
\begin{split}
C_\lambda &:=1 - \eta J^{(r,\lambda)}(0;1) - (1-\eta) J^{(r+m,\lambda)}(0;1) \\ &=  \eta\Big( 1- 
J^{(r,\lambda)}(0;1)
 \Big) + (1-\eta) \Big( 1- 
J^{(r+m,\lambda)}(0;1)
 \Big) =  \eta\Big( \frac{1+ \Phi(r)}{1 + \Phi(\lambda+ r)}  \Big) + (1-\eta) \Big(  \frac{1+ \Phi(r+m)}{1 + \Phi(\lambda+ r+m)} \Big).
 \end{split}
\end{align}
We need the following result in order to show the existence of $V_B^{*,\lambda}$. 
\begin{lemma}\label{Lem:Per:Aux1} The mapping $z \mapsto \Lambda^{(q,\lambda)}(z,z)$ is non-decreasing on $\R$ with the limit
\[
\lim_{z \downarrow -\infty} \Lambda^{(q,\lambda)}(z,z) = 
\begin{cases} 0 & \text{if } V_T > 0,\\
\frac{1}{q}\frac{\Phi(q)}{\Phi(\lambda + q)} & \text{if } V_T=0.
\end{cases}
\]
\end{lemma}
\begin{proof}
Suppose $V_T > 0$. By \eqref{fun_Lambda} together with the spatial homogeneity of the L\'evy process,
\begin{equation*}
\Lambda^{(q,\lambda)}(z,z) = \E_z \left[ \int_0^{T_z^-} e^{-qt} \Ind_{\lbrace X_t \geq \log V_T \rbrace} dt \right] = \E \left[ \int_0^{T_0^-} e^{-qt} \Ind_{\lbrace X_t \geq \log V_T - z \rbrace} dt \right].
\end{equation*}
As we can note from the previous identity, we have that the mapping $z\mapsto\Lambda^{(q,\lambda)}(z,z)$ is non-decreasing, and by bounded convergence $ \lim_{z\downarrow -\infty} \Lambda^{(q,\lambda)}(z,z) = 0 $.

On the other hand, for $V_T=0$ identity \eqref{fun_Lambda} implies that $\Lambda^{(q,\lambda)}(z,z) = \frac{1}{q} \left( 1 - J^{(q,\lambda)}(0;0) \right) = \frac{1}{q} \frac{\Phi(q)}{\Phi(\lambda + q)}$.
\end{proof}
\noindent This result leads to the following proposition, which describes the limiting behavior of $V_B \mapsto \Eq(V_B;V_B)$.
\begin{prop}\label{Prop:Per:Aux1} The mapping $V_B \mapsto \Eq(V_B;V_B)$ is strictly increasing on $(0,\infty)$ with the limits:
\begin{align*}
\lim_{V_B \downarrow 0} \Eq(V_B;V_B) &= 
\begin{cases}  - \frac{P(\rho + m)}{r +m }\frac{\Phi(r+m)}{\Phi(\lambda + r+m)}& \text{if } V_T > 0,\\
- P K_\lambda
& \text{if } V_T=0,
\end{cases}\\
\lim_{V_B \uparrow \infty} \Eq(V_B;V_B) &= \infty,
\end{align*}
	where
\begin{equation}\label{LT:Per:K}
K_\lambda :=  \frac{\rho + m}{r +m }\frac{\Phi(r+m)}{\Phi(\lambda + r+m)} - \frac{\kappa \rho}{r}\frac{\Phi(r)}{\Phi(\lambda + r)}.
\end{equation}
\end{prop}
\begin{proof}
Because  $J^{(q,\lambda)}(0;1) \in (0,1)$ for all $q\geq 0$, the term $C_\lambda$
is strictly positive. The proof follows by this fact, by Lemma \ref{Lem:Per:Aux1}, and the form of $\Eq(V_B;V_B)$ given by \eqref{Eq:Per:Diag}.
\end{proof}

\noindent Now by Proposition \ref{Prop:Per:Aux1}, 
we define our candidate optimal threshold $V_B^{*,\lambda}$ as follows.
\begin{enumerate}
\item If $V_T > 0$ or if $V_T=0$ with $K_\lambda > 0$, 
 we set $V_B^{*,\lambda}$ as the unique solution of \eqref{V_B_equality}.
\item Otherwise (i.e. $V_T=0$ with $K_\lambda  \leq 0$), 
we set $V_B^{*,\lambda} =0$.
\end{enumerate}
\subsection{Optimality of $V_B^{*\lambda}$}\label{subsec_optim_vb_lambda}
We now verify that the barrier $V_B^{*,\lambda}$ indeed maximizes the equity value subject to the limited liability constraint. To this end, we will verify the following.
\begin{enumerate}
\item Any level $V_B$ below $V_B^{*,\lambda}$ violates the limited liability condition  \eqref{limited_liability}.
\item $V_B^{*,\lambda}$ achieves a higher equity value than any level $V_B > V_B^{*,\lambda}$ does.
\item $V_B^{*,\lambda}$ fulfills the limited liability constraint  \eqref{limited_liability}.
\end{enumerate}
\begin{prop}\label{Prop:NLLC:Per} Suppose $V_B^{*,\lambda} > 0$. For $V_B < V_B^{*,\lambda}$ the limited liability constraint is not satisfied.
\end{prop}
\begin{proof}
By the (strict) monotonicity as in Proposition \ref{Prop:Per:Aux1} and because  $\mathcal{E}(V_B^{*,\lambda}; V_B^{*,\lambda}) =0$ (given that $V_B^{*,\lambda} > 0$), we have $\mathcal{E}(V_B ;V_B) < 0$ for $V_B < V_B^{*,\lambda}$, and hence the constraint fails to hold.
\end{proof}
\noindent The following results describe the behavior of the mapping $V_B \mapsto \Eq(V;V_B)$ on $(V_B^{*,\lambda},V)$.
\begin{lemma}\label{Prop:EqDeriv:Per} For $V_B \in (0,V)$, we have
\[
\frac{\partial \Eq}{\partial V_B}(V;V_B) = -\frac{1}{V_B} \left( \Phi(\lambda + r +m) - \Phi(r +m) \right) \left( \frac{V_B}{V} \right)^{\Phi(r+m)} L(V, V_B),
\]
where, 
for $x \geq z >0$,
\begin{equation*}
\begin{split}
L(x, z) &:= \frac{\Phi(\lambda + r) - \Phi(r)}{\Phi(\lambda + r +m) - \Phi(r +m)} \left( \frac{x}{z} \right)^{\Phi(r+m)-\Phi(r) } \left[ P\kappa\rho \Lambda^{(r,\lambda)}(\log z, \log z) + \eta z \left( \frac{1 + \Phi(r)}{1 + \Phi( \lambda + r  )} \right) \right] \\
& \quad + (1 -\eta) z \left( \frac{1 + \Phi(r+m)}{ 1+ \Phi( \lambda + r + m )} \right) - \frac{P(\rho + m)}{r +m }\frac{\Phi(r+m)}{\Phi(\lambda + r+m)}.
\end{split}
\end{equation*}

\end{lemma}
\begin{proof}
First, for $x \geq 0$ 
 and $q,\beta \geq 0$, the derivative of \eqref{1st:FI:Per} is given by
\begin{equation}\label{Opt:Per:Aux0}
J^{(q,\lambda) \prime}(x;\beta) = -\frac{ \Phi(q) (\Phi(\lambda + q) - \Phi(q))}{\beta + \Phi(\lambda + q)} e^{-\Phi(q) x},
\end{equation}
and hence
\begin{equation}\label{Opt:Per:Aux1}
J^{(q,\lambda)}(x;\beta) - J^{(q,\lambda) \prime}(x;\beta) = (\Phi(\lambda + q) - \Phi(q)) e^{- \Phi(q) x} \left( \frac{ 1 + \Phi(q)}{\beta + \Phi(\lambda + q)} \right).
\end{equation}
From the definition of $Z^{(q)}(\cdot ; \theta)$ given in \eqref{2nd:Scale:Fn}, we get the identity
\begin{equation}\label{2nd:Scale:Deriv}
Z^{(r) \prime}(x;\Phi(\lambda+r)) = \Phi(\lambda+r) Z^{(r)}( x ;\Phi(\lambda+r)) -\lambda W^{(r)}( x ).
\end{equation}
Hence, by differentiating \eqref{Res:Dens:Tax}, 
substituting \eqref{2nd:Scale:Deriv}, and grouping terms, we get 
\begin{align}
\frac{\partial }{\partial z} \Lambda^{(r,\lambda)}(x,z) &= -(\Phi(\lambda + r) - \Phi(r))\frac{\Phi(r)}{\Phi(\lambda +r)} e^{-\Phi(r)(x-z)} \left( \frac{1}{r} + \overline{W}^{(r)}( \log V_T - z) \Ind_{\lbrace V_T > 0 \rbrace} \right) \nonumber\\
& +   \frac{1}{\lambda}\frac{(\Phi(\lambda + r) - \Phi(r))^2}{\Phi(\lambda + r)} e^{-\Phi(r)( x-z)} Z^{(r)}( \log V_T - z;\Phi(\lambda + r)) \Ind_{\lbrace V_T > 0 \rbrace} \nonumber\\
 &= - (\Phi(\lambda + r) - \Phi(r)) e^{-\Phi(r)( x-z)} \Lambda ^{(r,\lambda)}(z,z), \label{Opt:Per:Aux2}
\end{align}
where the last equality holds by Remark \ref{remark_Lambda}.
Now, from expression \eqref{Eq:Per:Obs}, and by the chain rule we have
\begin{align*}
\frac{\partial \Eq}{\partial V_B}(V;V_B) &= \frac{1}{V_B} P\kappa\rho \frac{\partial \Lambda^{(r,\lambda)}}{\partial z} (\log V,\log V_B) - \eta \left( J^{(r,\lambda)}\left(\log \frac{V}{V_B};1\right) - J^{(r,\lambda) \prime} \left( \log \frac{V}{V_B};1 \right) \right)\\
&\quad - (1-\eta) \left( J^{(r+m,\lambda)}\left(\log \frac{V}{V_B};1\right) - J^{(r+m,\lambda) \prime} \left( \log \frac{V}{V_B};1 \right) \right) \\
&\quad - \frac{1}{V_B} \frac{P(\rho + m)}{r+m} J^{(r+m,\lambda) \prime}\left(\log \frac{V}{V_B};0\right).
\end{align*}
The result now follows by substituting expressions \eqref{Opt:Per:Aux0}--\eqref{Opt:Per:Aux2} for $q = r, r+m$. 
\end{proof}
\noindent The next result shows that $V_B^{*,\lambda}$ attains a higher equity value than any other $V_B > V_B^{*,\lambda}$ does.
\begin{prop}\label{Prop:Deriv:Opt} 
	For $V_B \in (V_B^{*,\lambda},V)$, we have $\frac{\partial \Eq}{\partial V_B} (V;V_B)< 0$. Hence $\Eq (V;V_B) < \Eq(V;V_B^{*,\lambda})$ for all $V > V_B$. 
\end{prop}
\begin{proof}
In view of Lemma \ref{Prop:EqDeriv:Per} and since the term $\frac{1}{V_B} \left( \Phi(\lambda + r +m) - \Phi(r +m) \right) \left( \frac{V_B}{V} \right)^{\Phi(r+m)}$ is non-negative
it is sufficient to show that
\[
L( V,V_B) > 0.
\]
Because $\Phi$ is strictly increasing and concave (see, e.g., Appendix C.3 in \cite{palmowski2019leland}) and $V>V_B$ 
\[ \frac{\Phi(\lambda + r) - \Phi(r)}{\Phi(\lambda + r +m) - \Phi(r +m)} \left( \frac{V_B}{V} \right)^{ \Phi(r)-\Phi(r+m)} > 1, \]
and  hence, with \eqref{Eq:Per:Diag} and \eqref{C_lambda}, 
\begin{align*}
L( V,V_B) & > P\kappa\rho \Lambda^{(r,\lambda)}(\log V_B,\log V_B) - \frac{P(\rho + m)}{r +m }\frac{\Phi(r+m)}{\Phi(\lambda + r+m)} + V_B C_\lambda =  \Eq(V_B;V_B),
\end{align*}
which is positive by Proposition \ref{Prop:Per:Aux1}. This completes the proof.
\end{proof}
Now it is left to show that the optimal barrier $V_B^{*,\lambda}$ satisfies the \textit{limited liability condition}.
\begin{lemma}\label{Prop:Deriv:Feas} For $V > V_B > 0$ with $V \neq V_T$, we have
\begin{align*}
\frac{\partial \Eq}{\partial V}(V;V_B) &= 1 - \frac{V_B}{V} \left[ \frac{\partial \Eq}{\partial V_B}(V;V_B) + \eta J^{(r,\lambda)}\left( \log \frac{V}{V_B};1 \right) + (1 - \eta) J^{(r +m,\lambda)}\left( \log \frac{V}{V_B};1 \right) \right] \\
&\quad + \frac{1}{V} P\kappa\rho R^{(r,\lambda)}\left( \log \frac{V_B}{V}, \log\frac{V_B}{V_T} \right) \Ind_{\lbrace V_T > 0 \rbrace},
\end{align*}
where $R^{(r,\lambda)}$ is defined in 
Appendix \ref{app_fi}. 
\end{lemma}
\begin{proof}
We write the equity value \eqref{Eq:Per:Obs} as
\[
\Eq(V;V_B) = A(\log V,\log V_B) + P\kappa\rho \Lambda(\log V,\log V_B),
\]
where 
\[
A(x,z) := e^{x} - e^z  \big(\eta J^{(r,\lambda)}(x-z;1) + (1 -\eta ) J^{(r+m,\lambda)}(x-z;1)\big) - \frac{P(\rho + m)}{r +m }(1 - J^{(r+m,\lambda)}(x-z;0)).
\]
Here, the following relationship holds
\begin{equation}\label{Feas:Aux1}
\frac{\partial A}{\partial x}(x,z) = e^{x} - \frac{\partial A}{\partial z}(x,z) - e^z \left( \eta J^{(r,\lambda)}(x-z;1) + (1 -\eta )J^{(r+m,\lambda)}(x-z;1) \right).
\end{equation}
On the other hand, for $x > z$ with $x\not=\log V_T$, differentiating \eqref{Res:Dens:Tax} and by \eqref{Opt:Per:Aux2} and \eqref{Res:Dens:PO}, 
\begin{align}
\frac{\partial \Lambda^{(r,\lambda)}}{\partial x}(x,z) &= \frac{\Phi(r)}{\Phi(\lambda + r)}(\Phi(\lambda + r) - \Phi(r))e^{-\Phi(r)(x-z)} \left( \frac{1}{r} + \Ind_{\lbrace V_T > 0 \rbrace} \overline{W}^{(r)}(\log V_T -z) \right) \nonumber\\
&+ \frac{\Phi(r)}{\lambda\Phi(\lambda + r)}(\Phi(\lambda + r) - \Phi(r))e^{-\Phi(r)(x-z)} Z^{(r)}(\log V_T -z;\Phi(\lambda +r)) \Ind_{\lbrace V_T > 0 \rbrace} \nonumber \\
&-W^{(r)}(\log V_T -x)\Ind_{\lbrace V_T > 0 \rbrace} \nonumber\\
&= - \frac{\partial \Lambda^{(r,\lambda)}}{\partial z}(x,z) + R^{(r,\lambda)}(z-x,z-\log V_T ) \Ind_{\lbrace V_T > 0 \rbrace}. \label{Feas:Aux2}
\end{align}
Now the result follows by combining expressions \eqref{Feas:Aux1} and \eqref{Feas:Aux2}, followed by applying the chain rule. 
\end{proof}
\begin{prop}\label{Prop:Feas:Per} The barrier $V_B^{*,\lambda}$ satisfies the limited liability constraint.
\end{prop}
\begin{proof}
(i) First, we suppose that $V_B^{*,\lambda} > 0$, and take $V > V_B^{*,\lambda}$. By combining Proposition \ref{Prop:Deriv:Opt} and Lemma \ref{Prop:Deriv:Feas} and using that $J^{(q)} \left( \log \frac{V}{V_B^{*,\lambda}};1 \right) < 1$ in view of \eqref{1st:FI:Per}, and that $R^{(r,\lambda)}$ is the resolvent density (which is nonnegative), we get for $V\not=V_T$
\[
\frac{\partial \Eq}{\partial V}(V;V_B^{*,\lambda}) \geq 1 - \frac{V_B^{*,\lambda}}{V} \geq 0.
\]
The claim follows by applying this, the continuity of the mapping $V\mapsto\Eq(V;V_B^{*,\lambda})$, and by using that $\Eq(V_B^{*,\lambda};V_B^{*,\lambda})=0$.

(ii)  On the other hand when $V_B^{*,\lambda} = 0$, which in turn implies that $V_T=0$, we have
\begin{align*}
\Eq(V;0) &=  V + P \Big( \frac{\kappa \rho}{r} - \frac{\rho + m}{r+m} \Big),\\
\frac{\partial \Eq}{\partial V}(V;0) &= 1 >0.
\end{align*}
By the convexity of $\psi$ on $[0,\infty)$, we have $\Phi(r+m) > \Phi(r) > 0$ and $\delta_1 := \Phi(\lambda + r+m) - \Phi(r+m) < \Phi(\lambda + r) - \Phi(r) =: \delta_2$. Hence
\[
\frac {\Phi(\lambda + r+m)} {\Phi(r+m)} = 1 + \frac {\delta_1} {\Phi(r+m)} < 1 + \frac {\delta_2} {\Phi(r)}  = \frac {\Phi(\lambda + r)} {\Phi(r)},
\]
and therefore $\frac  {\Phi(r+m)} {\Phi(\lambda + r+m)} > \frac {\Phi(r)} {\Phi(\lambda + r)}$.
Now
\[
0 \leq \frac{\kappa \rho}{r}\frac{\Phi(r)}{\Phi(\lambda + r)}  - \frac{\rho + m}{r +m }\frac{\Phi(r+m)}{\Phi(\lambda + r+m)} \leq \frac{\Phi(r)}{\Phi(\lambda + r)} \Big( \frac{\kappa \rho}{r} - \frac{\rho + m}{r +m } \Big),
\]
where the first inequality holds because $V_B^{*,\lambda} = 0$ also means $K_\lambda \leq 0$. This implies $ \frac{P \kappa \rho}{r} - \frac{P(\rho + m)}{r+m} \geq 0$, and hence $\Eq(V;0)$ is nonnegative. 
\end{proof}
\noindent Our main result follows from Propositions \ref{Prop:NLLC:Per}, \ref{Prop:Deriv:Opt},   and \ref{Prop:Feas:Per}. 
\begin{theo} \label{theorem_main_periodic}
The optimal bankruptcy level in the periodic case is given by $V_B^{*,\lambda}$.
\end{theo}
\section{Two-stage problem}\label{Sec:TwoStage}

We turn our attention to the problem of determining an optimal value of the face value of debt $P$, say $P^*$, such that the 
firm value is maximized. As pointed out by Chen and Kou \cite{ChenKou09}, this problem is entangled with the problem of determining the optimal bankruptcy-triggering level. 
 We can see this 
in the expressions derived in Subsections \ref{SubSec:Opt:Cont} and \ref{SubSec:Opt:Per}, where $V_B^*$ and $V_B^{*,\lambda}$ can be seen as functions of $P$. 
Here we will find $V_B^*$, $V_B^{*,\lambda}$ and $P^*$ 
according to a two-stage optimization problem similar to those in  Leland \cite{Leland94} and Leland and Toft \cite{LelandToft96}. In our case, we will solve the two-stage optimization problem both in the continuous- and in the periodic-observation cases. Recently, Palmowski et al. \cite{palmowski2019leland} solved it for the spectrally negative case under periodic observations.

The first-stage optimization is the problem of determining the optimal bankruptcy-triggering level which maximizes the equity value, subject to the limited liability constraint. 
From the calculations up to this point, we note that such level depends on the debt level $P$, and hence we will write $V_B^*(P)$ and $V_B^{*,\lambda}(P)$, respectively. Once this bankruptcy level is determined for each $P$, the firm will conduct a second-stage optimization in order to maximize its firm value in terms of $P$. This second-stage optimization is formulated as
\begin{align}
\max_{P} \V(V;V_B^*(P),P), \label{V_maximum}
\end{align}
where we modify the notation of $\V$ to emphasize the dependency on $P$.
\begin{assum} 
	In order to obtain semi-explicit solutions for the two-stage problem, we assume throughout the rest of this section that $V_T=0$, as assumed in \cite{ChenKou09} and \cite{palmowski2019leland}.
\end{assum}

\subsection{Continuous-observation case}
We first consider the continuous-observation case. Here, recall  $K$ defined in \eqref{LT:Cont:K}. By the following theorem, there exists an optimal value of $P^*$ that attains \eqref{V_maximum}.

\begin{theo}[Continuous Observations] \label{theorem_two_stage_continuous}
 In the case of continuous observations, we have the following.
\begin{itemize}
\item[(i)] \textbf{First-stage optimization}. (a) If $K
\leq 0$,  then $V_B^*(P) = 0$ for all $P > 0$.  (b)  Otherwise, define  $\epsilon := \frac{K}{C
}$. Given a debt level $P$, the optimal bankruptcy barrier is given by $V_B^*(P) = \epsilon P$.
\item[(ii)] \textbf{Second-stage optimization}.  (a) If $
K
 \leq 0$, then $\V(V;V_B^*(P),P) = V + \frac{P\kappa\rho}{r}$, which is linear (and hence convex) in $P$.   (b)  Otherwise, for any $V>0$, the firm value $\V(V;V_B^*(P),P)$ is strictly concave in $P$ on $[0, \frac V {\epsilon}] = \{ P \geq 0: V_B^{*}(P) \leq V \}$.
\end{itemize}
\end{theo}
\begin{proof}
(a) Suppose $K \leq 0$. Recalling the assumption that $V_T=0$, we know from Section \ref{SubSec:Opt:Cont} that for this case $V_B^*(P)=0$ for all $P > 0$, showing (i). Moreover, substituting $V_B^*(P)=0$ into the value of the firm given by \eqref{Market_Value_Firm}, yields that $\V(V;0,P) = V + \frac{P\kappa\rho}{r}$.
 This shows (ii).

(b)  Now, suppose $K > 0$. We have that $V_B^*(P)$ is strictly positive and by Remark \ref{Remark_f}, it satisfies the condition $V_B^*(P)=\frac{\partial}{\partial V}\Eq(V_B^*(P)+;V_B^*(P), P) =0$. Hence, from $f(V_B^*(P)) = 0$
we get
\begin{align}
V_B^*(P) = \epsilon P, \label{V_B_epsilon}
\end{align}
proving (i).

For (ii), by \eqref{V_B_epsilon}, $\{ P \geq 0: V_B^{*}(P) \leq V \} = [0, \frac V {\epsilon}]$. Differentiating \eqref{Mkt:Val} with respect to $P$ yields
\begin{align*}
\frac{\partial}{\partial P}\V(V;V_B^*(P),P) &= \frac{\kappa\rho}{r} \left( 1- \exp\left( -\Phi(r)\log\frac{V}{\epsilon P} \right) \right) - \frac{\kappa\rho \Phi(r)}{r} \exp\left( -\Phi(r)\log\frac{V}{\epsilon P} \right)\\
&\quad - \eta \epsilon (1+\Phi(r)) \exp\left( -\Phi(r)\log\frac{V}{\epsilon P} \right).
\end{align*}
This shows that $\frac{\partial}{\partial P}\V(V; V_B^*(P),P)$ is a strictly decreasing function of $P$, and hence $\V(V;V_B^*(P),P)$ is strictly concave for $P \in [0,\frac{V}{\epsilon}]$.
\end{proof}
\noindent 

\subsection{Periodic-observation case} 
The formulation of the solution to the two-stage problem for the periodic case is similar to Palmowski et al. \cite{palmowski2019leland}. However in our case we do not need to make any additional assumption on the jump measure to guarantee the existence of the solution.  Similar to Theorem  \ref{theorem_two_stage_continuous}, for the periodic-observation case, there exists an optimal value of $P^{*,\lambda}$ such that \eqref{V_maximum} is attained.

\begin{theo}[Periodic Observations] \label{theorem_two_stage_periodic}
 In the case of periodic observations, we have the following:
\begin{itemize}
\item[(i)] \textbf{First-stage optimization}.  (a) If $K_{\lambda}\leq 0$, where $K_\lambda$ is defined in \eqref{LT:Per:K}, we have $V_B^{*,\lambda}(P)=0$ for all $P > 0$.  (b) Otherwise, with $\epsilon_{\lambda} := \frac{K_{\lambda}}{
C_\lambda
}$ 
, then $V_B^{*,\lambda}(P) = \epsilon_{\lambda} P$ for all $P >0$.
\item[(ii)] \textbf{Second-stage optimization}.  (a) If $K_{\lambda}\leq 0$, then $\V(V;V_B^{*,\lambda}(P),P) = V + \frac{P\kappa\rho}{r}$, which is linear (and hence convex) in $P$.  (b) Otherwise, for any $V>0$, the firm value $\V(V;V_B^{*,\lambda}(P),P)$ is a strictly concave function in $P$ on $[0, \frac V {\epsilon_\lambda}] = \{ P \geq 0: V_B^{*,\lambda}(P) \leq V \}$.
\end{itemize}
\end{theo}
\begin{proof} 
(a) The fact that $V_B^{*,\lambda}(P)= 0$ for the case  $K_{\lambda}\leq 0$, follows by construction as in Section \ref{SubSec:Opt:Per}. Moreover, by substituting $V_B^*(P)=0$ into expression \eqref{Val:Per} we get
$\V(V;0,P) = V+ \frac{P \kappa\rho}{r}$.

 (b) Now we consider the case  $K_{\lambda}> 0$.  Recall that $V_B^{*,\lambda}(P)$ satisfies the \textit{continuous fit} condition \[
 \Eq(V_B^{*,\lambda}(P);V_B^{*,\lambda}(P),P) = 0.
 \] Hence, by setting \eqref{Eq:Per:Diag} equal to 0 and using \eqref{Res:Dens:Tax}, we obtain
\begin{equation*}
0 = V_B^{*,\lambda}(P)
C_\lambda
  + \frac{P\kappa \rho}{r} \left( 1 - J^{(r,\lambda)}\left( 0 ;0 \right) \right) - \frac{P(\rho + m)}{r +m }\frac{\Phi(r+m)}{\Phi(\lambda + r+m)},
\end{equation*}
which can be re-arranged as
\begin{align}
V_B^{*,\lambda}(P) = \epsilon_\lambda P, \label{V_B_epsilon_lambda}
\end{align}
proving (i). By \eqref{V_B_epsilon_lambda}, $\{ P \geq 0: V_B^{*,\lambda}(P) \leq V \} = [0, \frac V {\epsilon_\lambda}]$. 

For (ii),
we obtain from expression \eqref{Val:Per} together with \eqref{Res:Dens:Tax} and \eqref{V_B_epsilon_lambda}, that the firm value is given by 
\[
\V(V; V_B^{*,\lambda}(P),P) = V + P \frac{\kappa\rho}{r} \left( 1 - J^{(r,\lambda)}\left( \log\frac{V}{ \epsilon_\lambda P} ;0 \right) \right) - \eta  \epsilon_\lambda P J^{(r,\lambda)}\left( \log\frac{V}{ \epsilon_\lambda P} ;1 \right).
\]
 By differentiating with respect to $P$ we get 
\begin{align*}
\frac{\partial}{\partial P} \V(V;V_B^{*,\lambda}(P),P) &= \frac{\kappa\rho}{r} \left( 1 - J^{(r,\lambda)}\left( \log\frac{V}{ \epsilon_\lambda P} ;0 \right) \right) + \frac{\kappa \rho}{r} J^{(r,\lambda)\prime} \left( \log\frac{V}{ \epsilon_\lambda P} ;0 \right) \\
&\quad - \left( \eta \epsilon_{\lambda} J^{(r,\lambda)}\left( \log\frac{V}{ \epsilon_\lambda P} ;1 \right)- \eta \epsilon_{\lambda}  J^{(r,\lambda)\prime} \left( \log\frac{V}{ \epsilon_\lambda P} ;1 \right) \right).
\end{align*}
Note that 
\[
J^{(r,\lambda)\prime}(y;\beta) = -\Phi(r) J^{(r)}(y;\beta),\quad \text{$\beta\geq 0$, $y > 0$.}
\]
Hence, with this fact and by grouping terms, we can write
\begin{align*}
\frac{\partial}{\partial P} \V(V;V_B^{*,\lambda}(P),P) &= \frac{\kappa\rho}{r} \left( 1 - J^{(r,\lambda)}\left( \log\frac{V}{ \epsilon_\lambda P} ;0 \right) \right)  - \frac{\kappa\rho \Phi(r)}{r} J^{(r,\lambda)}\left( \log\frac{V}{ \epsilon_\lambda P};0 \right)\\
& \quad - \eta \epsilon_{\lambda}  (1 + \Phi(r)) J^{(r,\lambda)}\left( \log\frac{V}{ \epsilon_\lambda P} ;1 \right).
\end{align*}
Since the mapping $y \mapsto J^{(r,\lambda)}(y;\beta)$ is strictly decreasing for all $\beta \geq 0$, we conclude that $ \frac{\partial}{\partial P} \V(V;V_B^{*,\lambda}(P),P) $ is strictly decreasing in $P$. Hence the value of the firm is a strictly concave function of $P$ on the interval $[0,\frac{V}{\epsilon_\lambda}]$, which proves our claim. 
\end{proof}
\section{Numerical Examples}\label{Sec:Num}


We conclude this paper through a sequence of numerical results, focusing on the case where 
$X$ is given by a mixture of Brownian motion and a compound Poisson process with i.i.d.\ phase-type distributed jumps:
\begin{align}
X_t= c t +\sigma B_t +\sum_{i=1}^{N_t}U_i, \quad t \geq 0, \label{X_phase}
\end{align}
where $(B_t)_{t \geq 0}$ is a standard Brownian motion, $(N_t)_{t \geq 0}$ is a Poisson process with intensity $\gamma$ and $(U_i)_{i\geq 1}$ takes a phase-type random variable that approximates a (folded) normal random variable with mean zero and variance $1$ (whose parameters are given in \cite{EGAMI2014}).

We use the same parameter sets as those used in  \cite{palmowski2019leland} (who use the same parameters as those in \cite{ChenKou09, HilbRog02, Kyprianou2007, Leland94, LelandToft96}): we set  $r=7.5\%$, $\delta=7\%$, $\kappa = 35\%$, $\eta =50\%$, $\rho = 8.162\%$, $m=0.2$, and $V_T = P\rho/\delta$.
Because $V_T > 0$, we must have $V_B^* > 0$ and  $V_B^{*,\lambda} > 0$.

The scale function for the \lev process of the form \eqref{X_phase} admits an explicit expression written as a sum of  exponential functions (usually with complex-valued coefficients); see e.g.\ \cite{EGAMI2014, Kuznetsov2013}.
In particular, we consider the following parameters: $\sigma = 0.2$, $c =  -0.24767$, $\gamma = 0.5$ ($c$ is chosen so that the martingale property $-\psi(1)=r-\delta = 0.005$ is satisfied). 
 Unless stated otherwise,  we set $P=50$ and, for the periodic case, $\lambda = 4$ (on average four times per year).




 \subsection{Optimality}\label{subsec:optimal} We first confirm the optimality of the suggested barrier $V_B^*$ and $V_B^{*, \lambda}$  for both continuous- and periodic-observation cases, which are given by
 \eqref{Cont:Equiv:FP} and  \eqref{V_B_equality}, respectively. Because these functions are monotone, we can apply classical bisection methods.
%
 The corresponding firm/debt/equity values can be computed  by \eqref{Mkt:Val}, \eqref{Debt:Val}, and \eqref{Eq:Val} for the continuous-observation case and \eqref{Val:Per}, \eqref{Debt:Per}, and \eqref{Eq:Per:Obs} for the periodic-observation case.

At the top of Figure \ref{figure_value}, for both continuous and periodic cases, we plot  $V \mapsto \mathcal{E}(V; V_B^*)$ and $V \mapsto \mathcal{E}(V; V_B^{*, \lambda})$ along with $V \mapsto \mathcal{E}(V; V_B)$ for $V_B \neq V_B^*$ and $V_B \neq V_B^{*,\lambda}$, respectively. The optimality as in Theorems \ref{theorem_main_classical} and \ref{theorem_main_periodic} can be confirmed. For the continuous-observation case, the level $V_B^*$ satisfies the limited liability constraint \eqref{limited_liability}, and any level $V_B$ lower than $V_B^*$ violates \eqref{limited_liability}. The same can be observed for the periodic case. We also confirm the smooth fit  for the continuous-observation case and continuous fit for the periodic-observation case.

\begin{figure}[htbp]
\begin{center}
\begin{minipage}{1.0\textwidth}
\centering
\begin{tabular}{cc}
\includegraphics[scale=0.5]{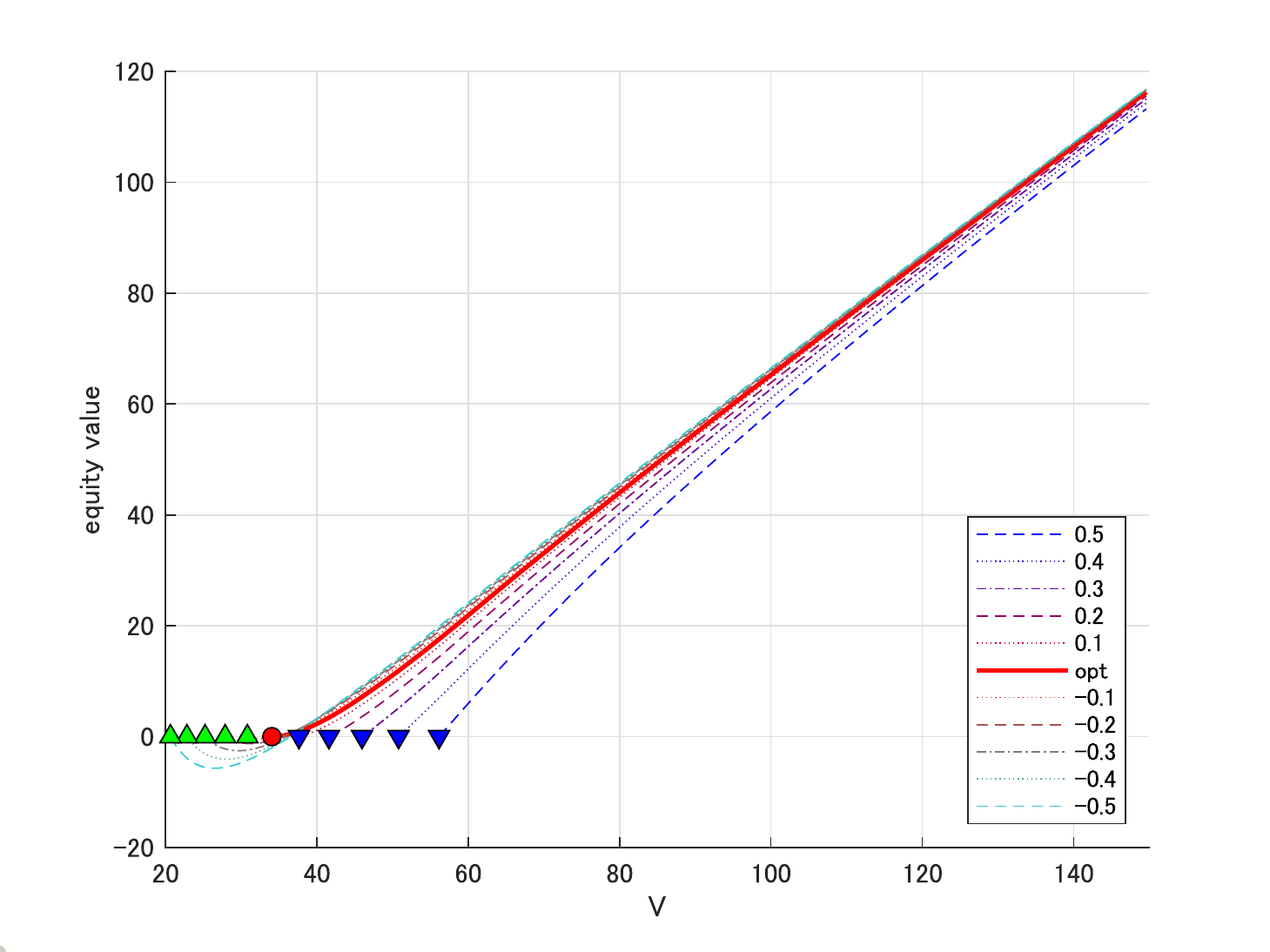} & \includegraphics[scale=0.5]{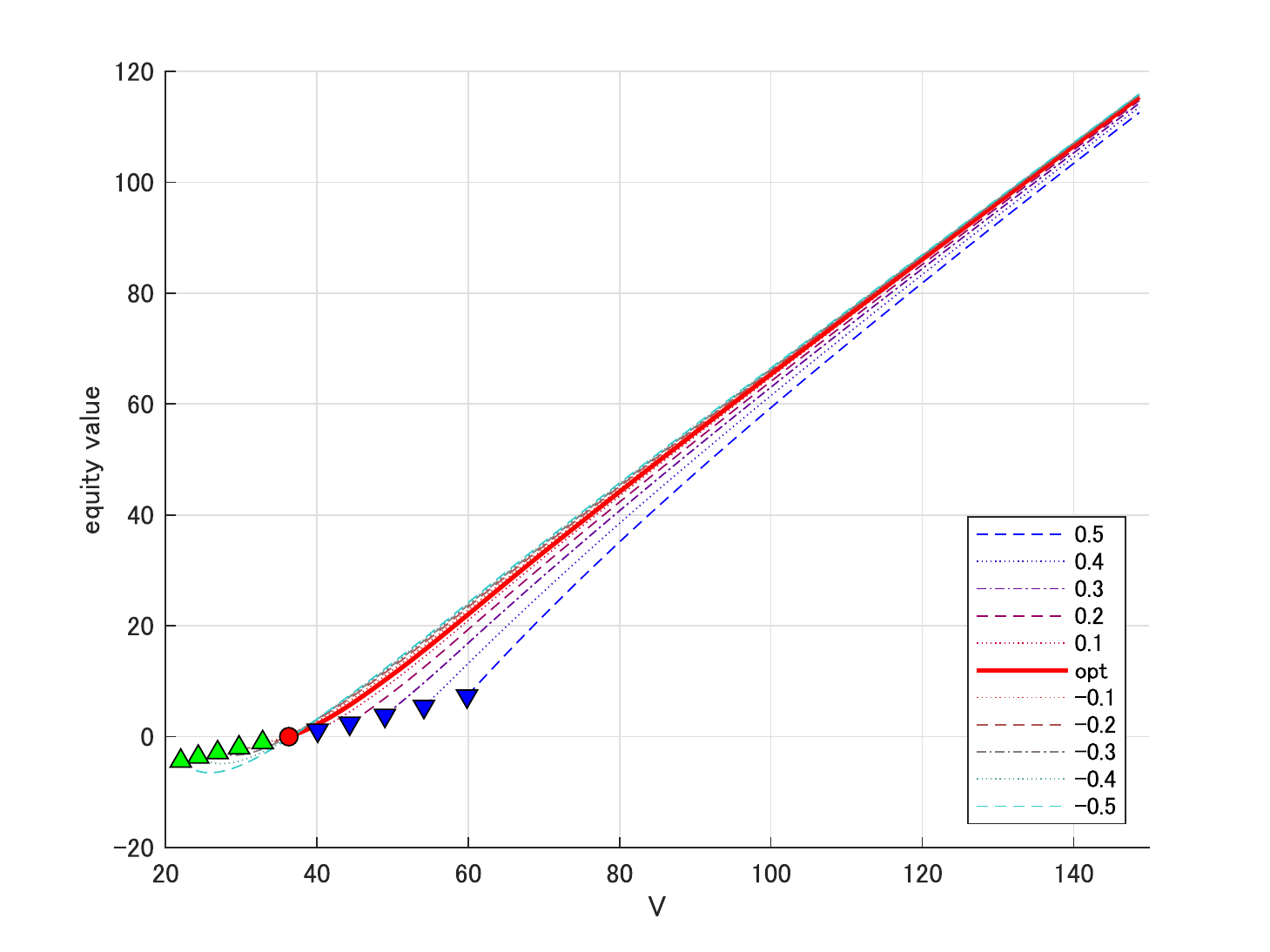} \\
\textbf{Continuous}: equity value $V \mapsto \mathcal{E}(V; V_B)$ & \textbf{Periodic}: equity value  $V \mapsto \mathcal{E}(V; V_B)$ \\
\includegraphics[scale=0.5]{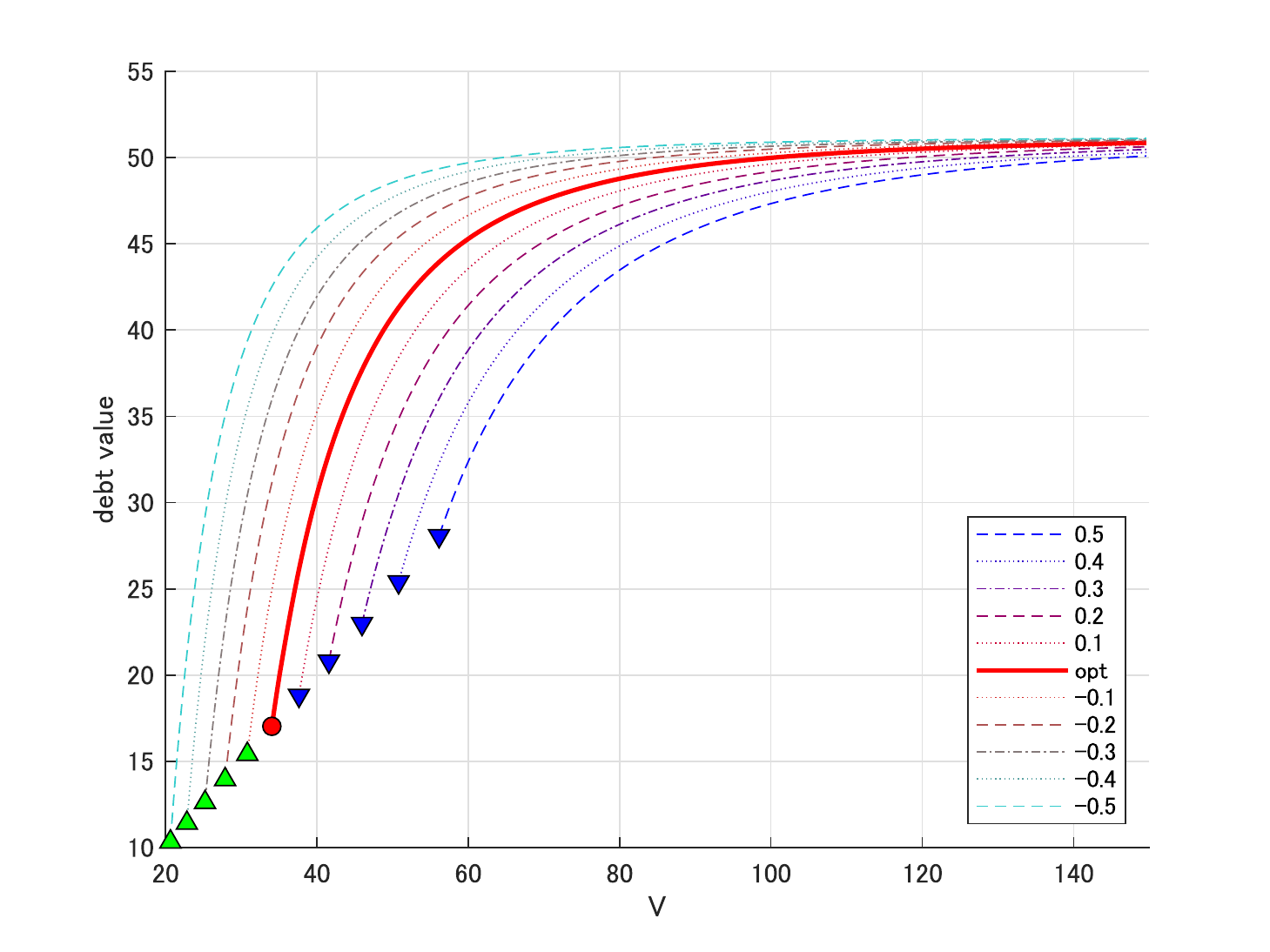} & \includegraphics[scale=0.5]{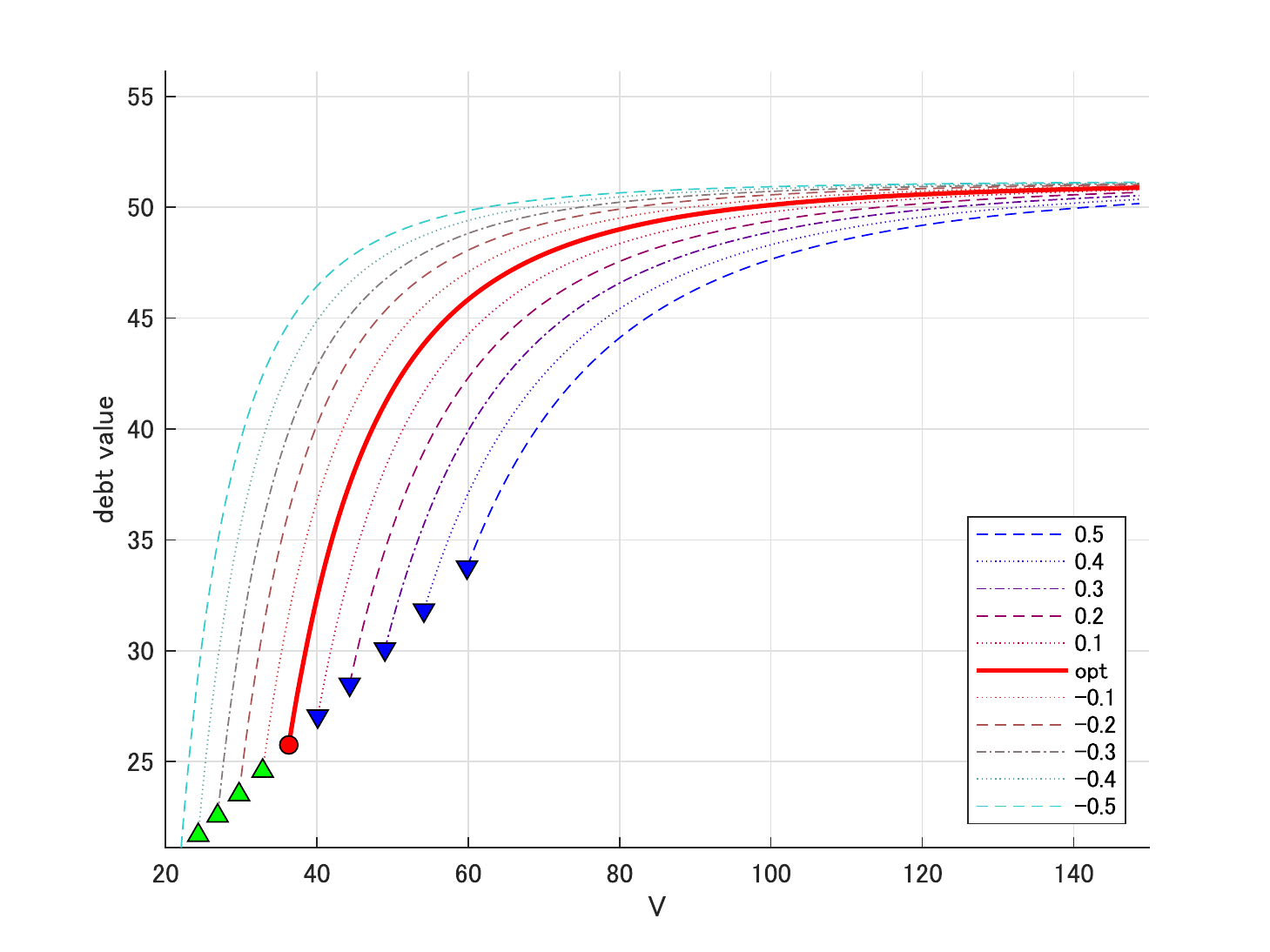} \\
\textbf{Continuous}: debt value  $V \mapsto \mathcal{D}(V; V_B)$ & \textbf{Periodic}: debt value  $V \mapsto \mathcal{D}(V; V_B)$ \\
\includegraphics[scale=0.5]{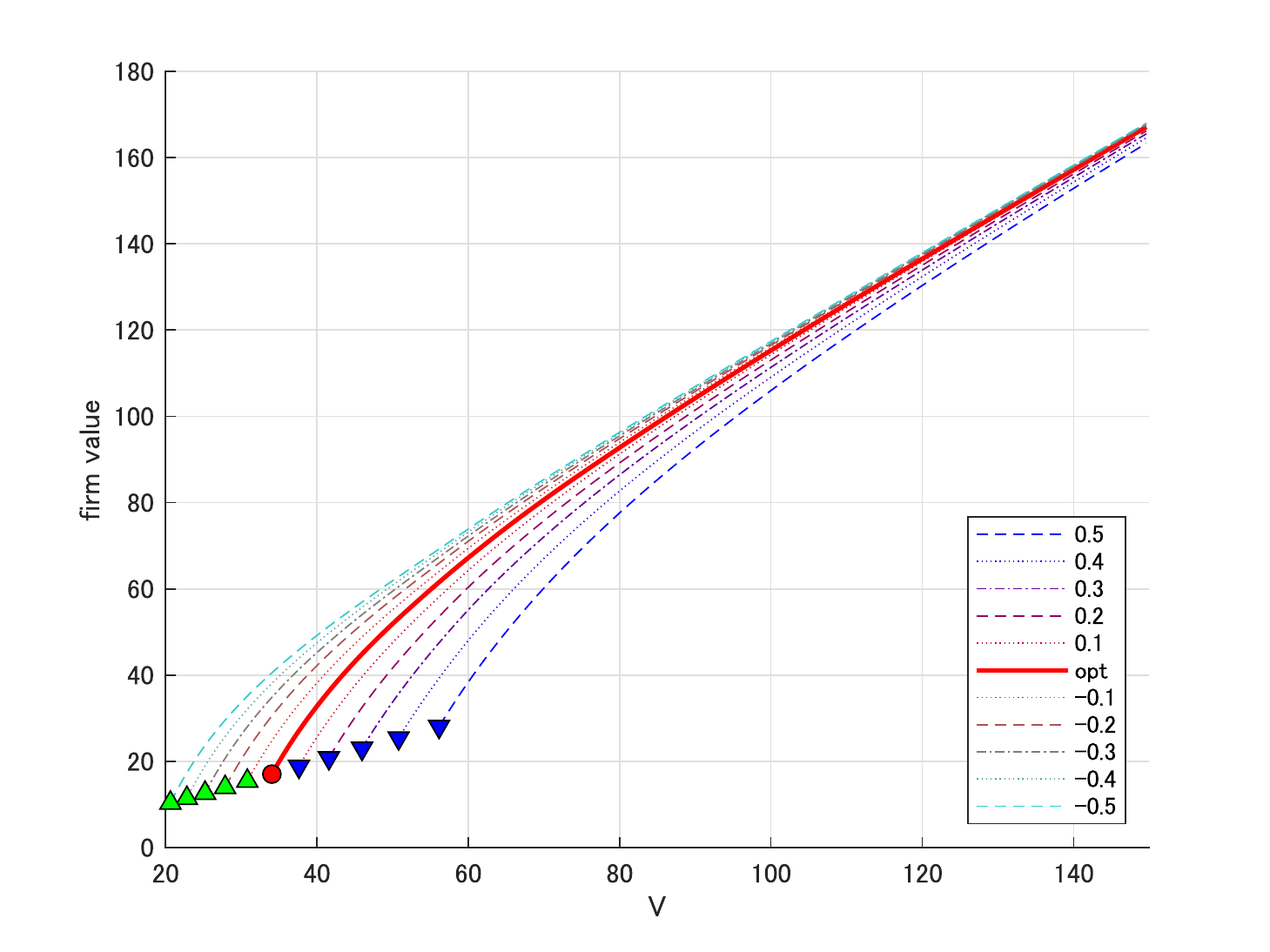} & \includegraphics[scale=0.5]{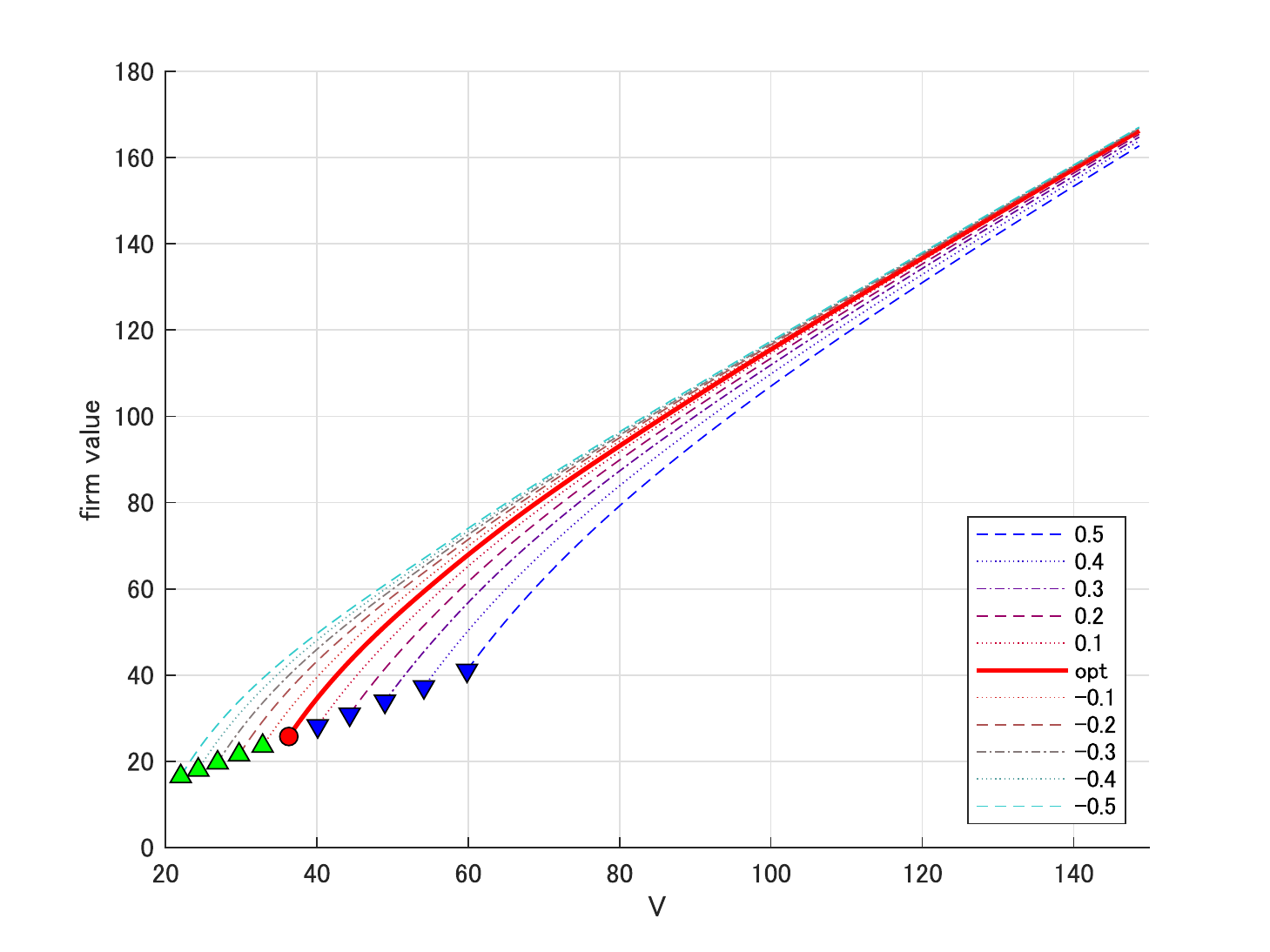} \\
\textbf{Continuous}: firm  value  $V \mapsto \mathcal{V}(V; V_B)$ & \textbf{Periodic}: firm value  $V \mapsto \mathcal{V}(V; V_B)$
\end{tabular}
\end{minipage}
\caption{(Left) The equity/debt/firm values as functions of $V$  on $(V_B, \infty)$ for $V_B=V_B^*$ (solid) along with $V_B = V_B^*  \exp(\epsilon)$ (dotted) for $\epsilon = -0.5,-0.4, \ldots, -0.1, 0.1, 0.2, \ldots, 0.5$. The values at $V = V_B$ are indicated by circles for $V_B = V_B^*$ and  those for  $V_B < V_B^*$ (resp.\ $V_B > V_B^*$) are indicated by up- (resp.\ down)-pointing triangles. (Right) Same results for the periodic case.} \label{figure_value}
\end{center}
\end{figure}


\subsection{Sensitivity with respect to $\lambda$ of the equity value} 
We next study the impacts of the rate of observation $\lambda$ on the optimal strategies. In Figure \ref{fig_r}, the equity value $\mathcal{E}(\cdot; V_B^{*,\lambda})$ is shown for various values of the observation frequency $\lambda$ along with those in the continuous-observation case. As $\lambda$ increases, the optimal barrier decreases and converges to that in the continuous observation case. The convergence of the equity value is also confirmed. These results confirm the discussions given in Remark \ref{remark_convergence}.


\begin{figure}[htbp]
\begin{center}
\begin{minipage}{1.0\textwidth}
\centering
\begin{tabular}{c}
\includegraphics[scale=0.5]{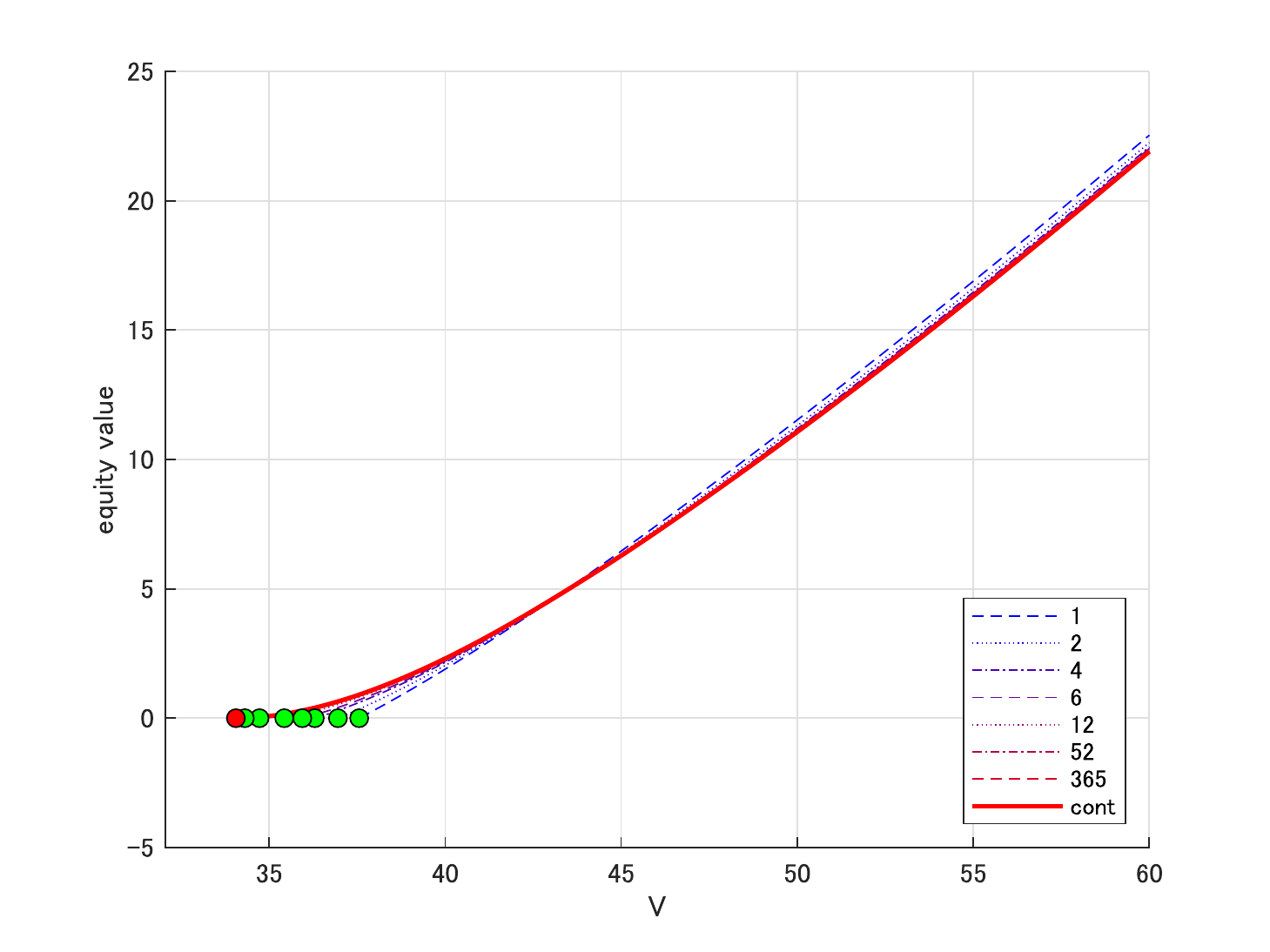} \\
\end{tabular}
\end{minipage}
\caption{\small{The equity values in the periodic-observation case $\mathcal{E}(V; V_B^{*,\lambda})$ (dotted) for $\lambda = 1,2,4,6,12,52,365$ along with the continuous-observation case   $\mathcal{E}(V; V_B^{*})$  (solid).  The corresponding values at $V = V_B^{*,\lambda}$ and $V = V_B^*$  are indicated by circles. 
}} \label{fig_r}
\end{center}
\end{figure}

\subsection{Two-stage problem}

\begin{figure}[htbp]
\begin{center}
\begin{minipage}{1.0\textwidth}
\centering
\begin{tabular}{cc}
\includegraphics[scale=0.5]{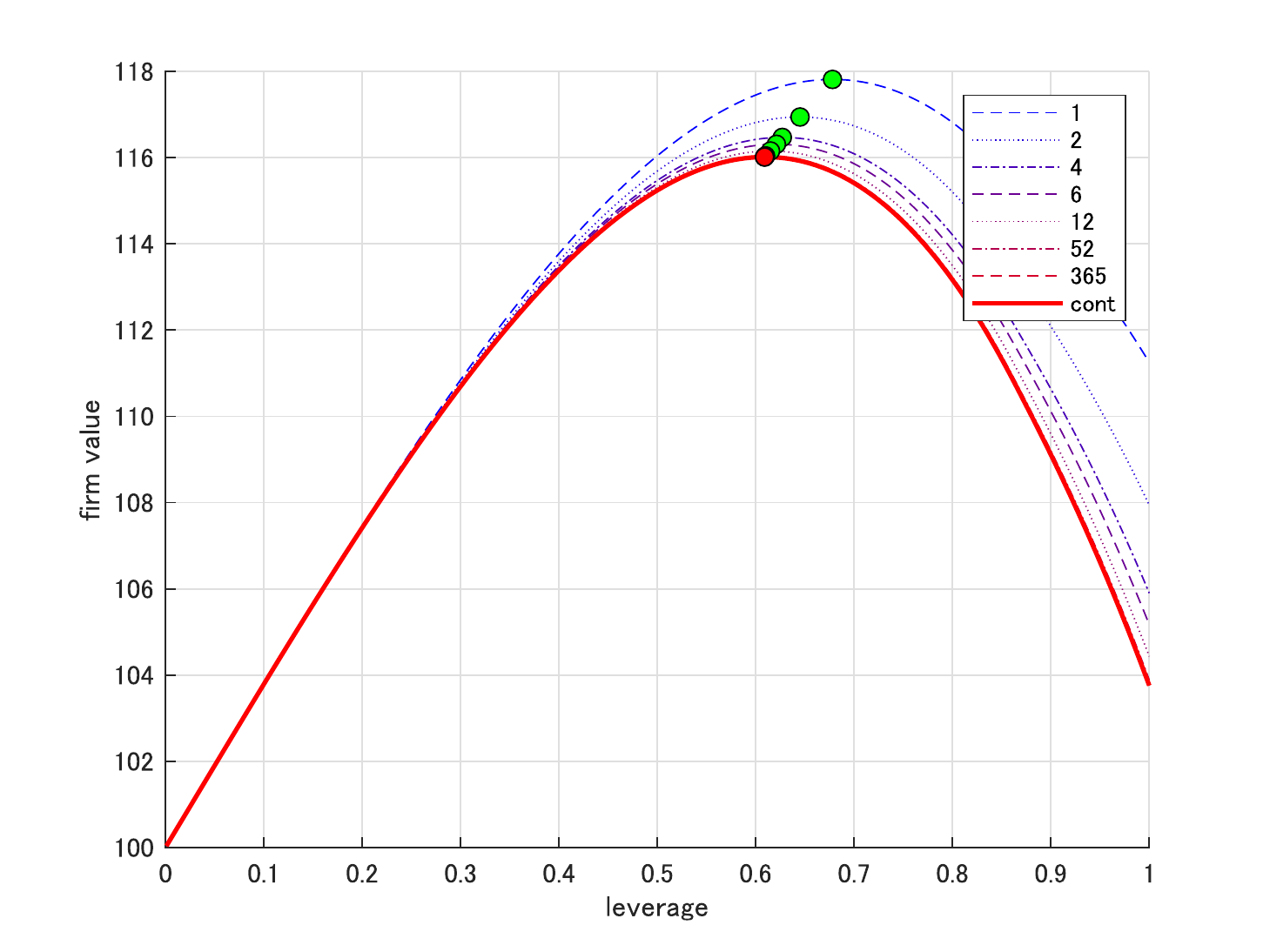}  & \includegraphics[scale=0.5]{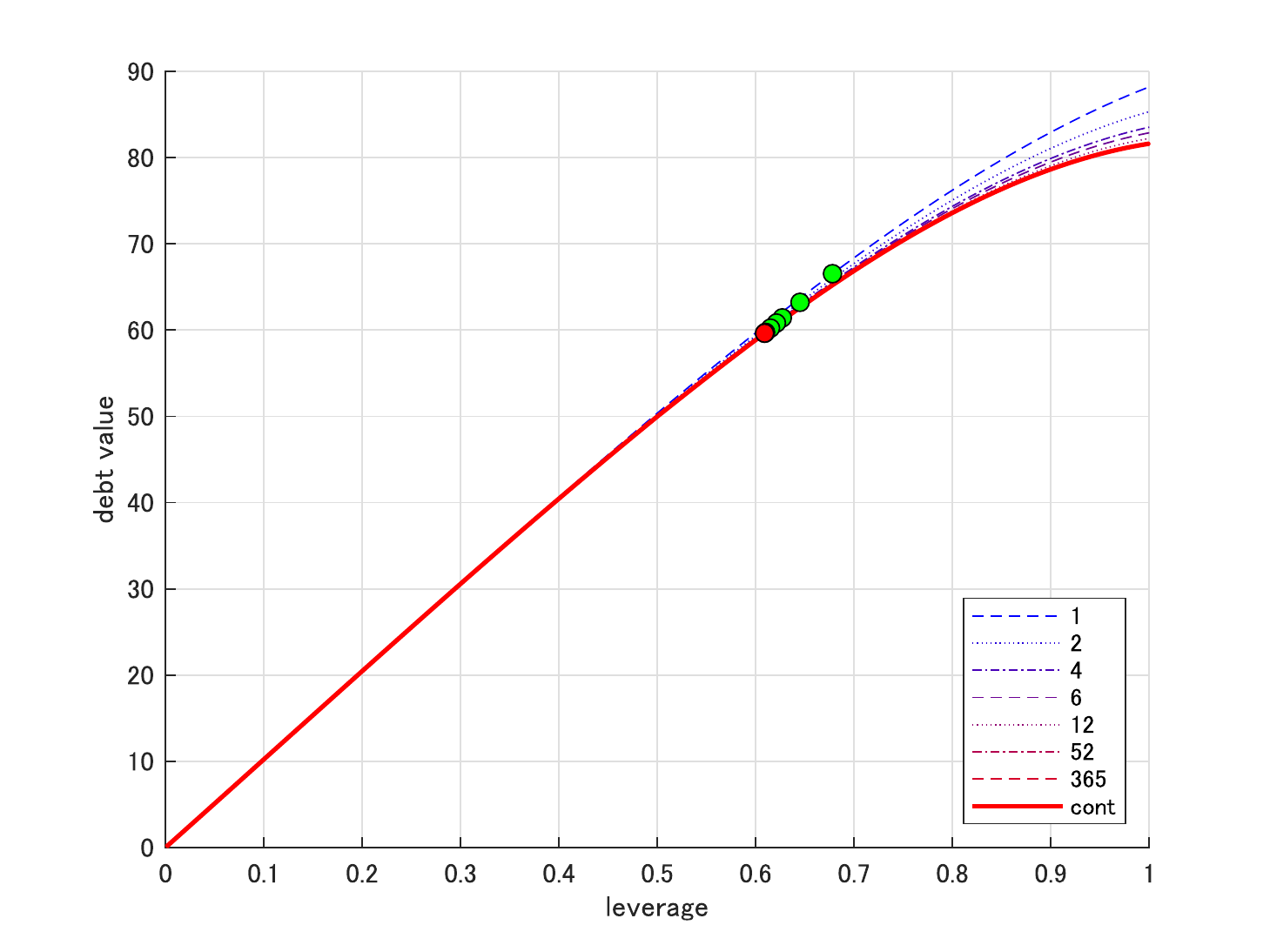}  \\
 firm value  &debt value  \\
\end{tabular}
\end{minipage}
\caption{\small{The firm (left) and debt (right) values with respect to  the leverage $P/V$ for the two-stage problem when $V = 100$. Those in the continuous-observation case are indicated by solid lines and periodic cases with $\lambda = 1,2,4,6,12,52,365$ are indicated by dotted lines. The points at the optimal leverage $P^*/V$ are shown by the circles.}} \label{figure_two_stage}
\end{center}
\end{figure}

We now move onto the two-stage problem as studied in Section \ref{Sec:TwoStage}. In Theorems \ref{theorem_two_stage_continuous} and \ref{theorem_two_stage_periodic}, the firm values $\mathcal{V}(V;V_B^*(P), P)$ and $\mathcal{V}(V;V_B^{*,\lambda}(P), P)$ have been confirmed to be concave in $P$ for the case $V_T = 0$.  To study if the same result holds when $V_T > 0$, 
 we keep using $V_T = P\rho/\delta$ as a function of $P$.

For $V = 100$,  we compute $V_B^*$ and  $V_B^{*,\lambda}$ for $P$ running from $0$ to $V=100$ (i.e.\ leverage $P/V$ running from $0$ to $1$) and then the corresponding firm and debt values for the continuous and periodic cases for various values of $\lambda$.  The firm and debt values are given in  Figure \ref{figure_two_stage}.
At least in the considered cases, the concavity with respect to $P$ is confirmed. In addition, monotonicity with respect to $\lambda$ of firm and debt values as well as the optimal barriers are also observed.

 \subsection{Sensitivity with respect to the jump rate}
 
 We now move onto analyzing the impact of the positive jumps of the process $X$ on the optimal bankruptcy levels as well as the optimal capital structure. To this end, we repeat the experiments conducted above for various values of  the jump rate $\gamma$ of positive jumps. In order to see easily the sensitivity with respect to $\gamma$, here we decide not to stick to the risk-neutral condition and simply change the value of $\gamma$ (without changing the drift).

In Figure \ref{figure_barrier_against_jump}, we plot the  equity values for various selections of $\gamma$ and the bankruptcy levels $V_B^*$ and $V_B^{*,\lambda}$ as functions of $\gamma$ for both continuous- and periodic-observation cases, respectively. Interestingly,  the equity value and the optimal bankruptcy level fail to be monotone. While this behavior with respect to $\gamma$ is not entirely intuitive, our interpretation is given as follows.

When $\gamma$ is sufficiently high, even when the current asset is low, one can expect sudden positive jumps to reach a healthy state above $V_T$ so that one can enjoy the tax shields, producing higher firm values. On the other hand, with low $\gamma$, this 
 is unlikely to be achieved -- as a result,  the firm value does not increase as rapidly as the debt value does, resulting in the decrease of the equity value. Because the optimal barrier is the root of $\Eq(V_B;V_B) = 0$, it tends to increase, in $\gamma$, when $\gamma$ is low but tends to decrease when $\gamma$ is high.
 
In Figure \ref{figure_two_stage_jumps}, we further study how $\gamma$ changes the optimal capital structures (similarly to what we studied in Figure \ref{figure_two_stage}). Contrary to what we observed in Figure \ref{figure_barrier_against_jump}, monotonicity with respect to $\gamma$ holds at least in this considered case. In particular, the firm and debt values both tend to increase in $\gamma$. In addition, the optimal leverage increases as $\gamma$ increases.
 
\begin{figure}[htbp]
\begin{center}
\begin{minipage}{1.0\textwidth}
\centering
\begin{tabular}{cc}
\includegraphics[scale=0.5]{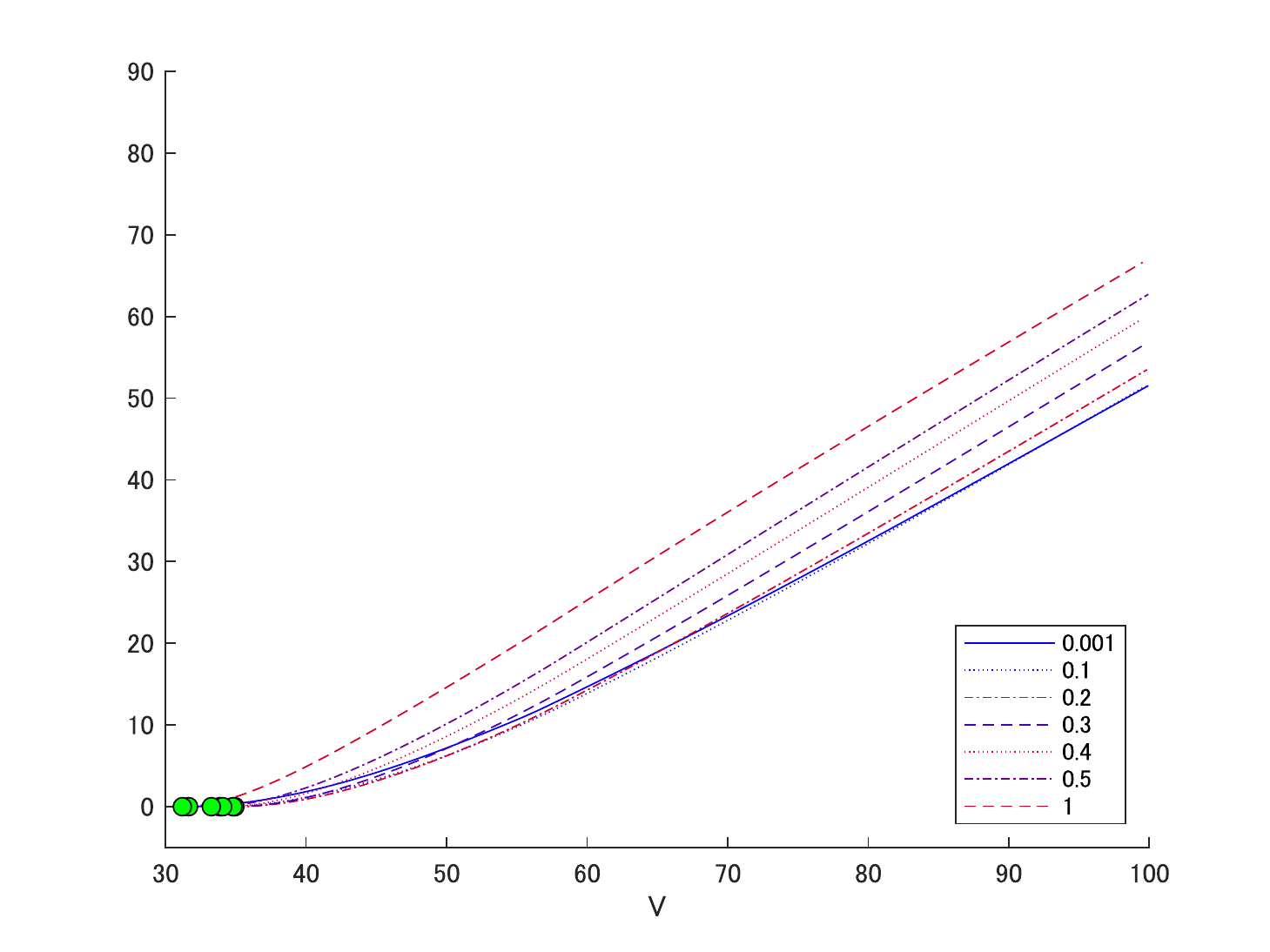}  & \includegraphics[scale=0.5]{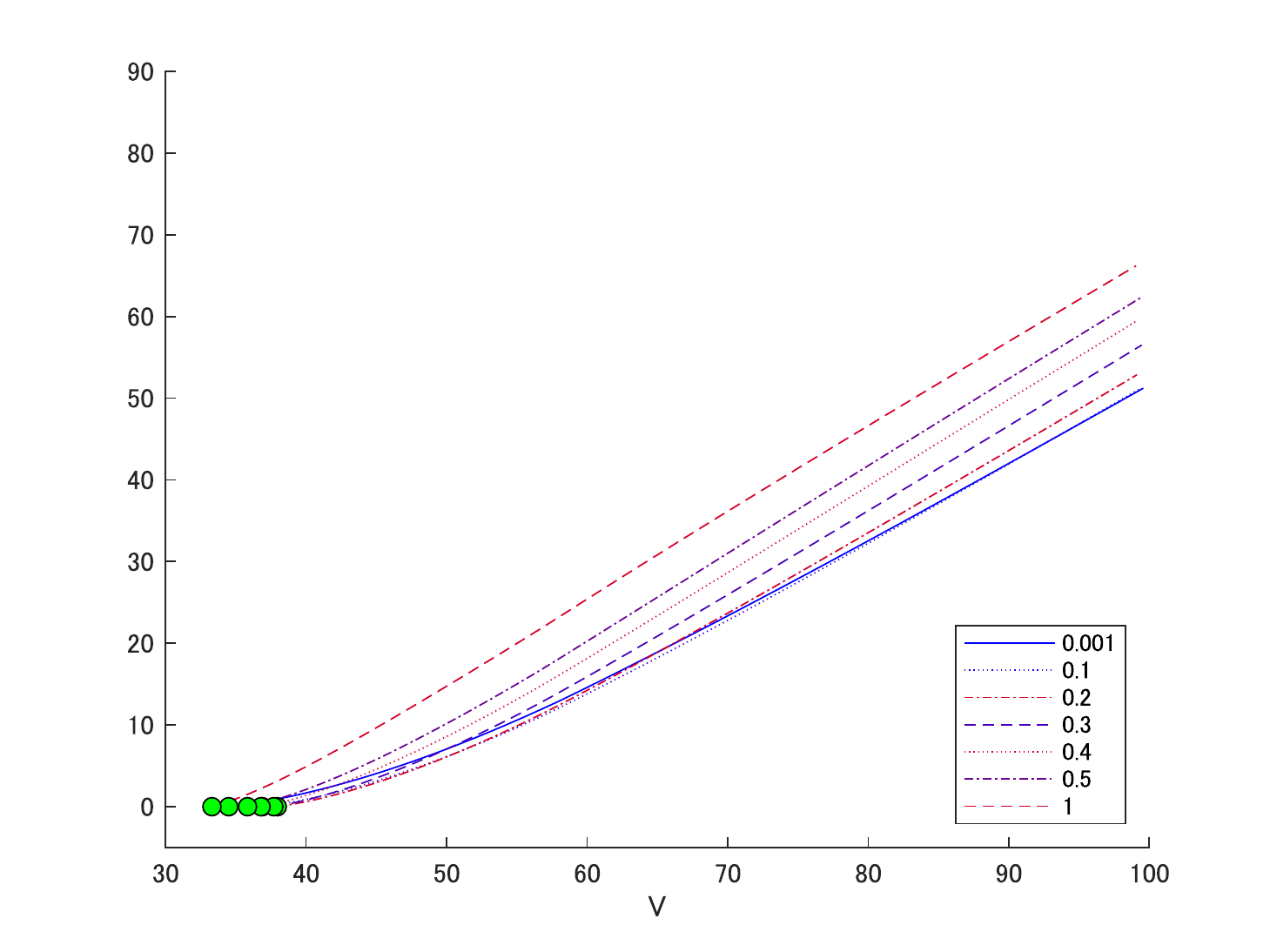}  \\ 
\includegraphics[scale=0.5]{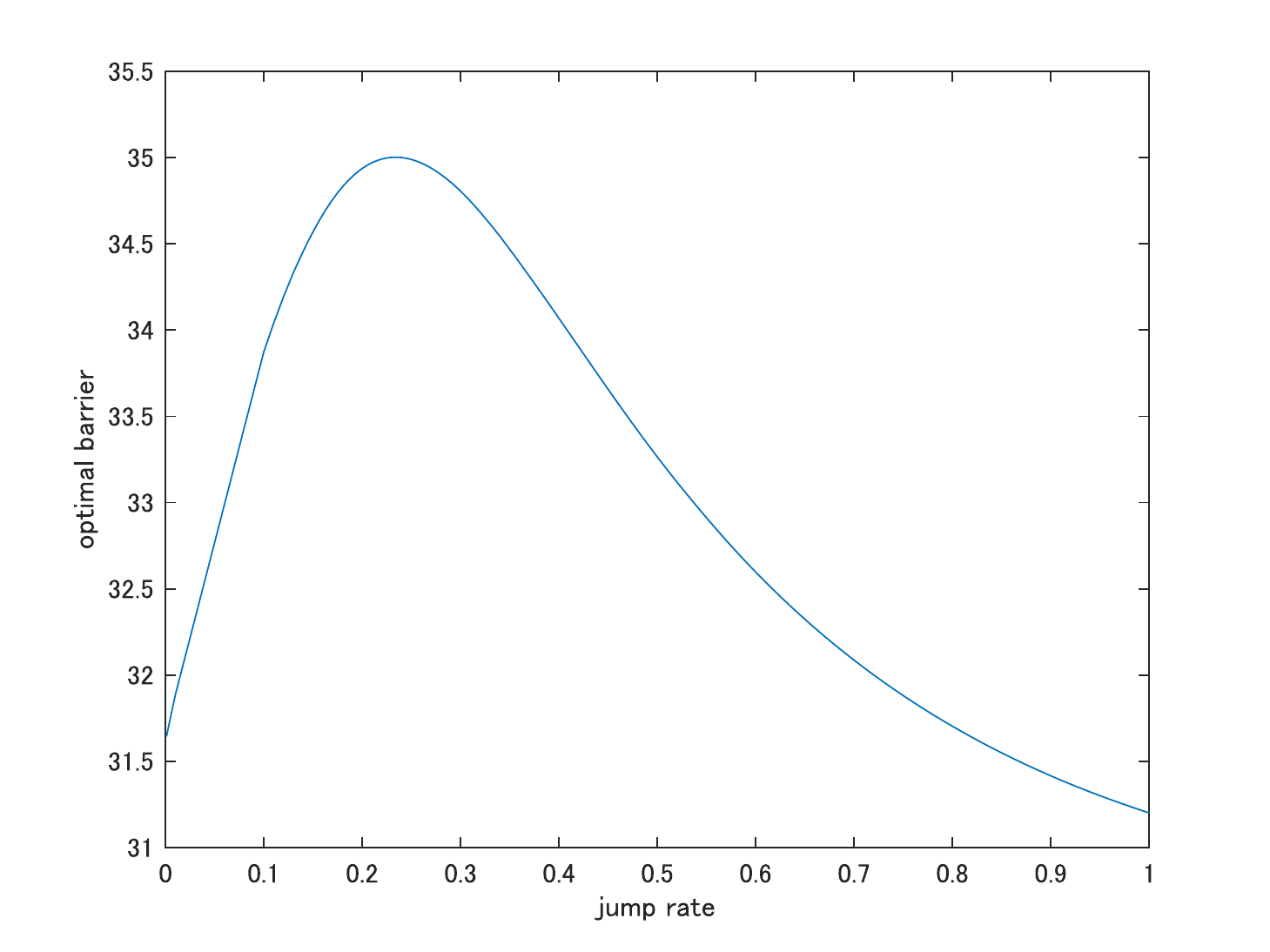}  & \includegraphics[scale=0.5]{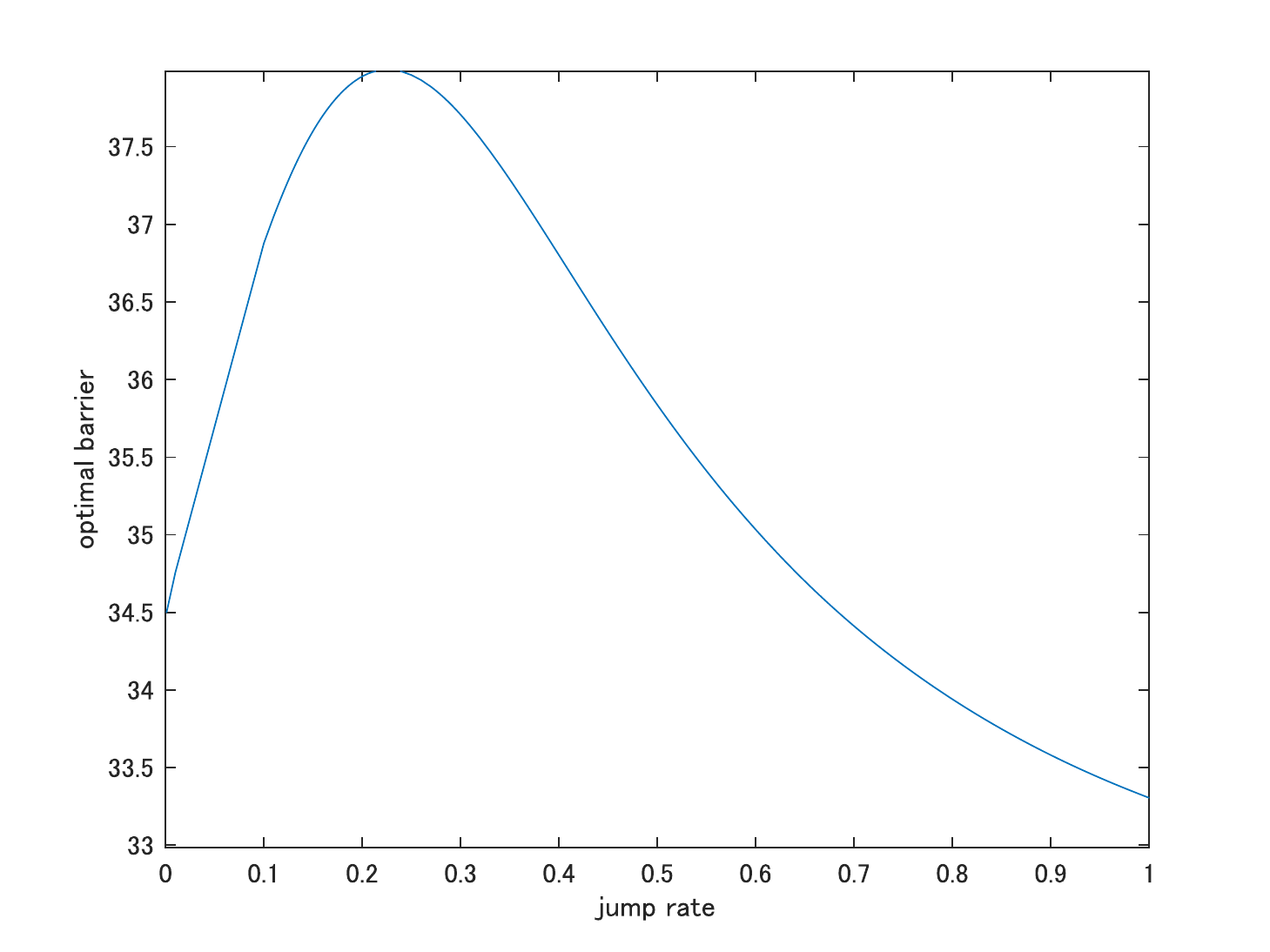}  \\
 continuous case & periodic case \\
\end{tabular}
\end{minipage}
\caption{\small{(Top) The equity values $\mathcal{E}(V; V_B^*)$ and $\mathcal{E}(V; V_B^{*,\lambda})$ for $\gamma = 0.001,0.1,0.2,0.3,0.4,0.5,1$ (left for the continuous-observation case and right for the periodic-observation case).  (Bottom) optimal barrier $V_B^*$ and $V_B^{*,\lambda}$ with respect to the jump rate $\gamma$.}} \label{figure_barrier_against_jump}
\end{center}
\end{figure}

\begin{figure}[htbp]
\begin{center}
\begin{minipage}{1.0\textwidth}
\centering
\begin{tabular}{cc}
\includegraphics[scale=0.5]{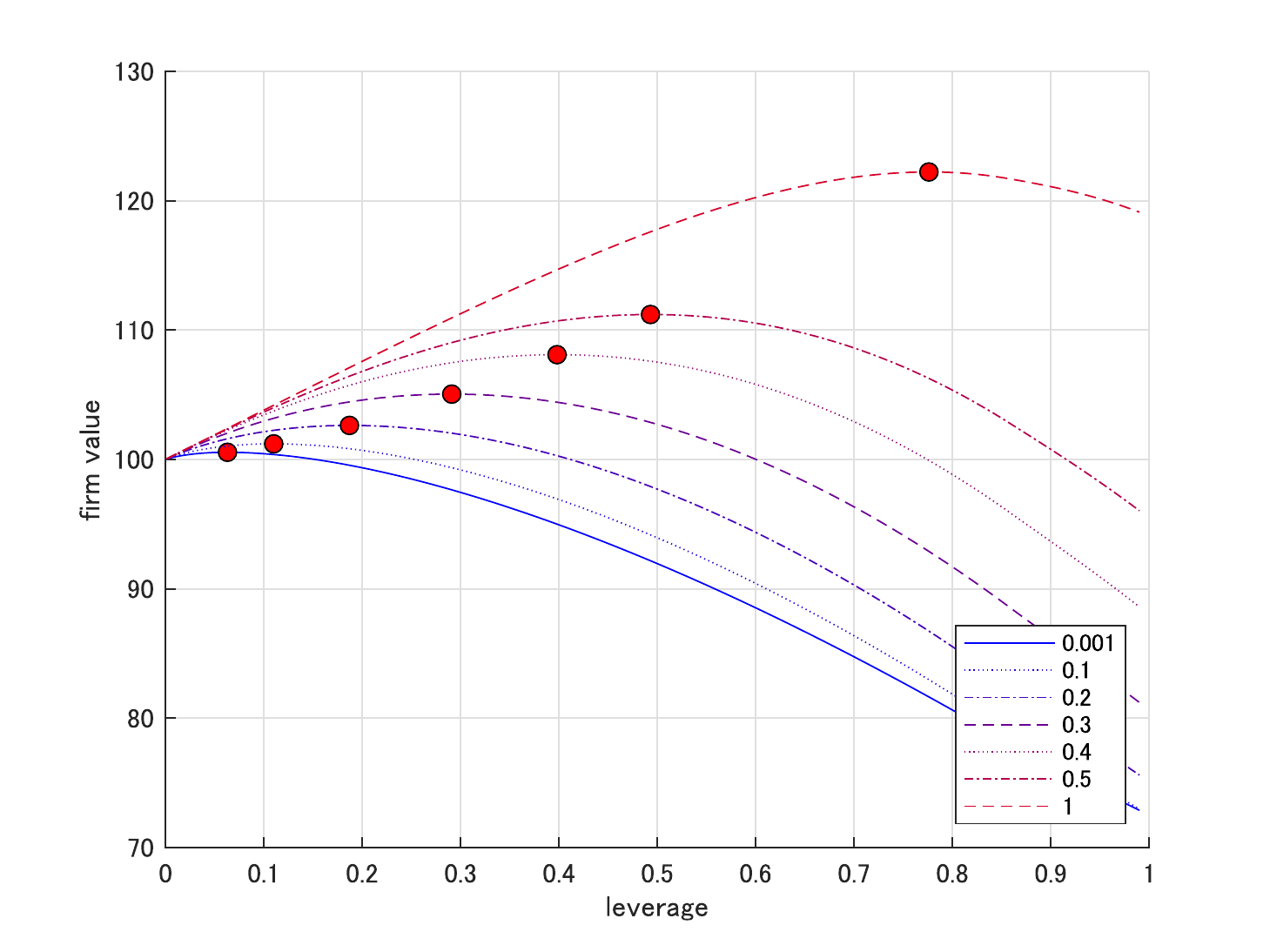}  & \includegraphics[scale=0.5]{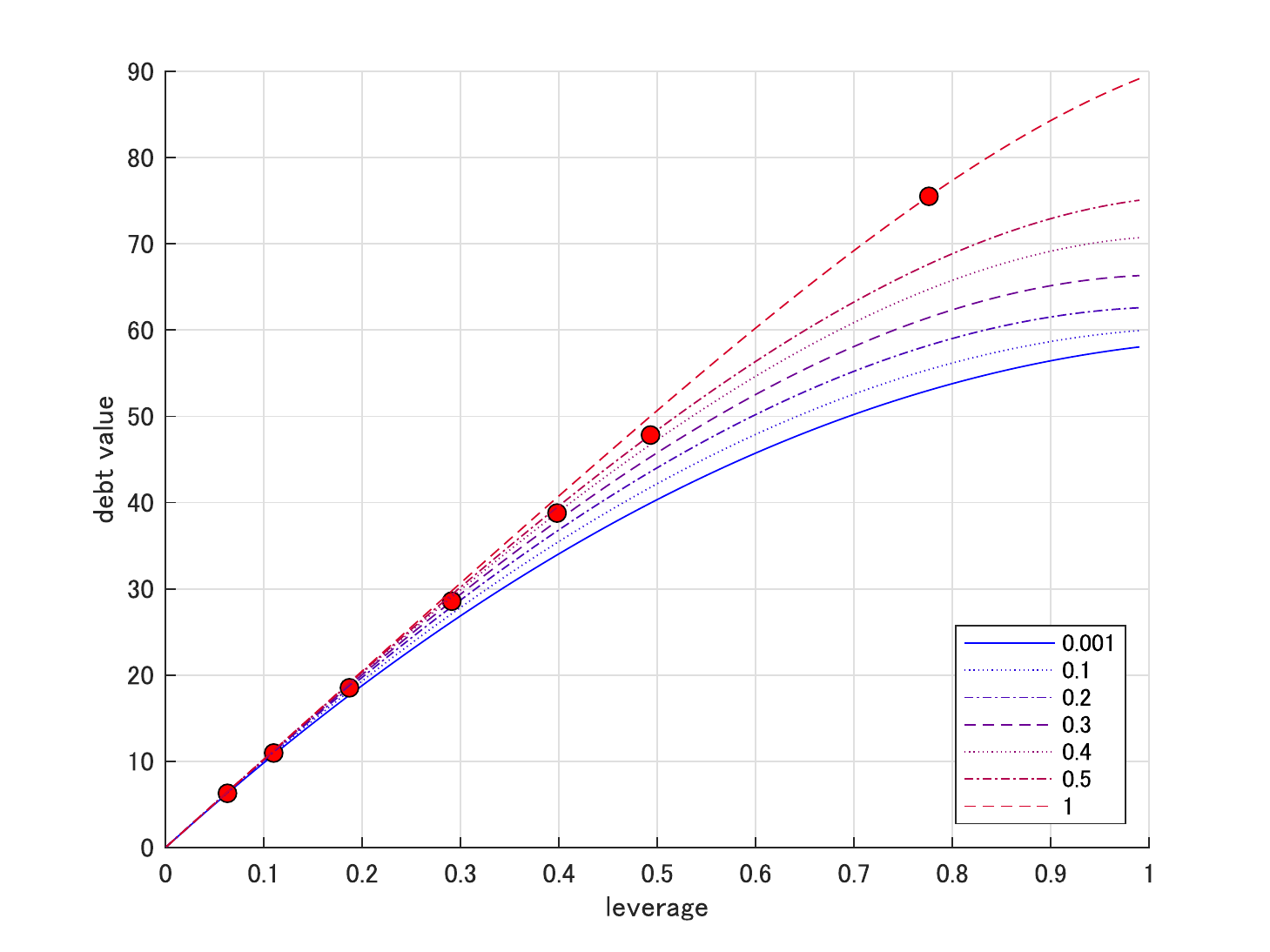} \\
\textbf{Continuous}:  firm value  & \textbf{Continuous}:  debt value  \\
\includegraphics[scale=0.5]{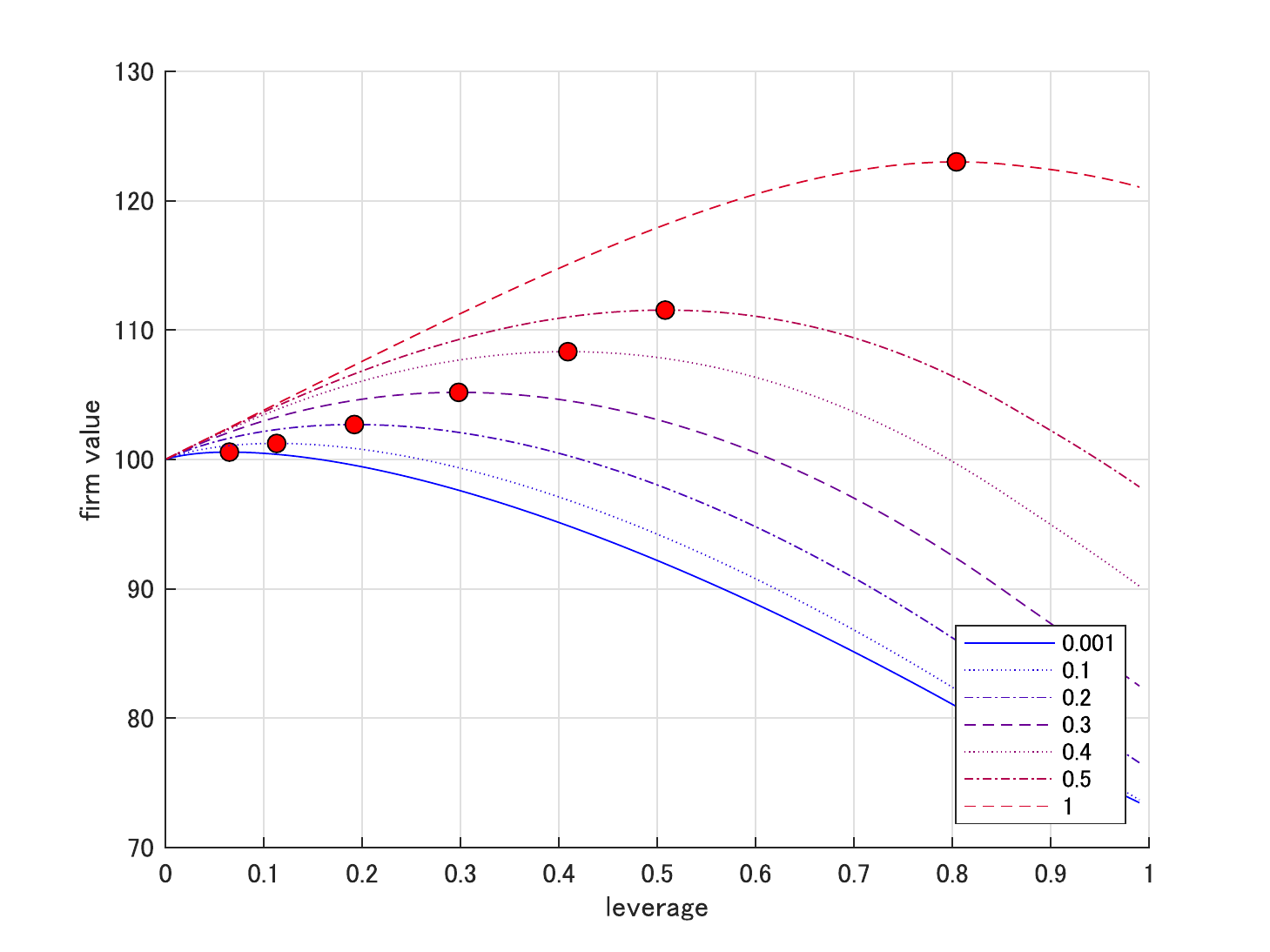}  & \includegraphics[scale=0.5]{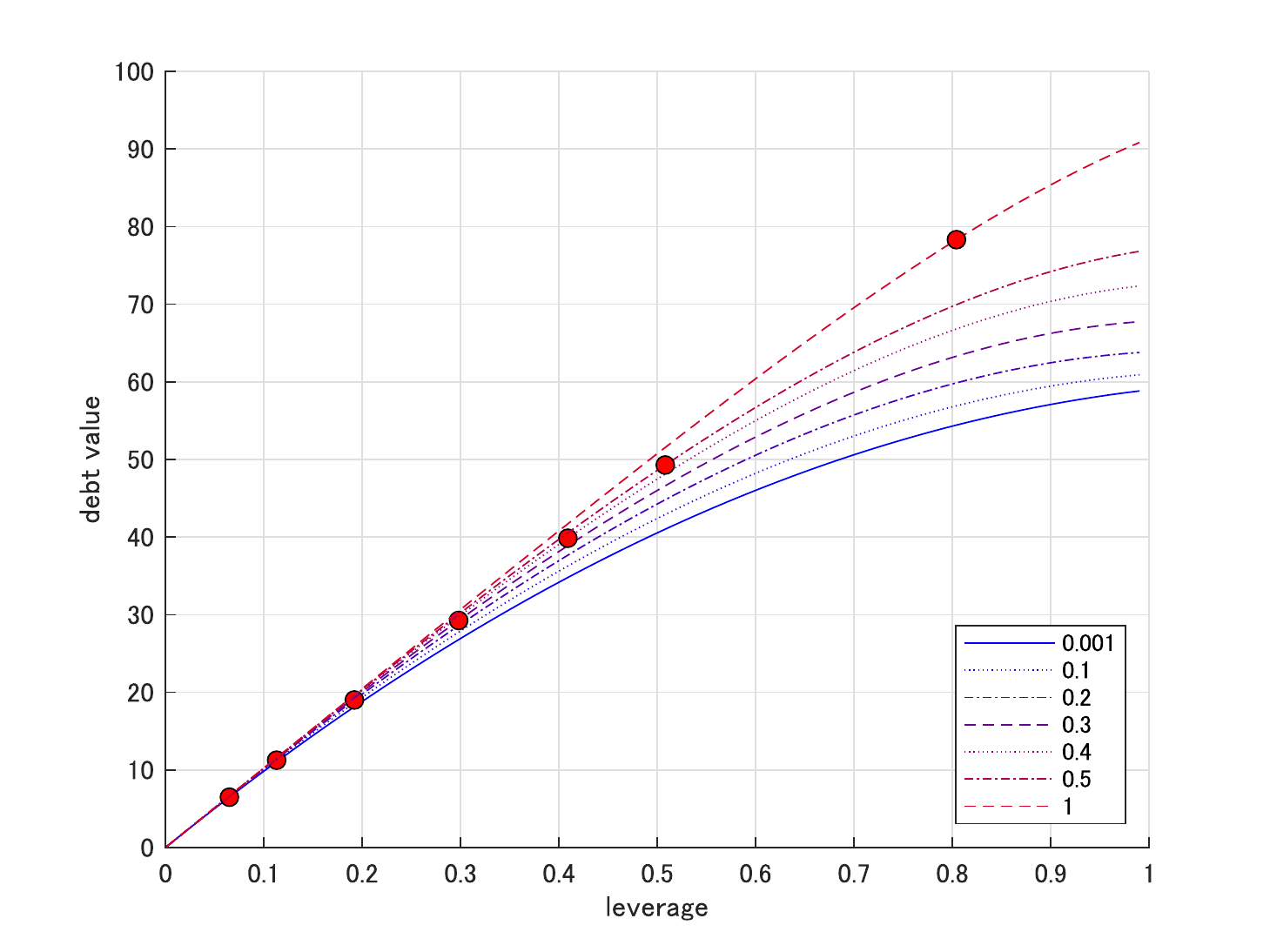}  \\
\textbf{Periodic}:  firm value  &\textbf{Periodic}: debt value  \\
\end{tabular}
\end{minipage}
\caption{\small{
The firm (left) and debt (right) values with respect to  the leverage $P/V$ for the two-stage problem when $V = 100$ for  $\gamma = 0.001,0.1,0.2,0.3,0.4,0.5,1$ (top for the continuous-observation case and bottom for the periodic-observation case). 
The points at the optimal leverage $P^*/V$ are shown by the circles.}}
 \label{figure_two_stage_jumps}
\end{center}
\end{figure}

\section{Concluding Remarks}\label{Sec:Final}


In this paper, we studied the Leland-Toft optimal capital structure model when the asset value follows a spectrally positive \lev process. We obtained explicitly the optimal endogenous bankruptcy level in both the continuous- and periodic-observation cases, written in terms of the scale function. The optimal capital structure was then obtained by considering the two-stage optimization problem. Our numerical results show that positive jumps of the asset value have significant impact on the firm's decision-making. In particular, the optimal bankruptcy barrier fails to be monotone in the rate of positive jumps.

There are various venues for future research. By considering positive jumps of the asset value, this paper complements the results obtained in the spectrally negative model as in \cite{HilbRog02, Kyprianou2007, surya2014optimal}. However, the assumption of no negative jumps has a drawback on the computation of credit spread,
it tends to converge to zero as the maturity approaches zero. In order to analyze the credit spread more accurately, as done in \cite{ChenKou09}, it is desirable to consider the cases with both positive and negative jumps. While a number of detailed analysis has been done in \cite{ChenKou09}, it is of great interest to consider more general jump distributions and also periodic-observation cases. This may be  possible for some \lev processes such as the phase-type \lev process \cite{asmussen2004russian} and the meromorphic \lev process \cite{kuznetsov2012meromorphic}.


It is also of interest to consider other models of bankruptcy. The bankruptcy in the periodic-observation model considered in this paper can be seen as the Parisian ruin with exponential delay as discussed in the introduction of \cite{palmowski2019leland}. Bankruptcy is triggered at the first time the asset value stays continuously below the bankruptcy barrier for an  independent exponential time.  In this setting, the ``distress level'' of the firm is reset to zero each time the asset value goes above the barrier (before the exponential clock rings). On the other hand, as considered in \cite{moraux2002valuing}, it is important to also consider the case the distress level accumulates (while the asset value is below the level) without being reset. For this extension, our approach using the fluctuation theory can potentially be used to obtain explicit solutions.

\appendix
\section{Review of fluctuation identities for spectrally negative L\'evy processes}\label{app_fi}

	In this section we review the fluctuation identities for \textit{spectrally negative} L\'evy processes. Let  $Y:=-X$ be the dual of $X$, and let $\tilde{\E}_x$ be the law of $Y$ when $Y_0 = x$ (in particular, $\tilde{\E} = \tilde{\E}_0$).

Let
\[ \tilde{\tau}_{-a}^{+} := \inf\lbrace t>0: Y_t > -a \rbrace,\qquad a\in\R.\]
By the duality relation between $Y$ and $X$, 
$\tilde{\tau}_{-a}^{+}$ under $\tilde{\E}$ has the same law as $\tau_a^-$ defined in \eqref{fptc} under $\E$.

We also recall the first passage time below a level $a$ under Poissonian observations defined in \eqref{Sp_Pos_Poisson_Passage}, which is equal in distribution to the first \textit{observed} passage time above level $-a$  of the dual process
\[
\tilde{T}_{-a}^+ = \inf\lbrace T \in \T: Y_T > -a \rbrace, \qquad a\in\R.
\]

Fix $q\geq0$. The first identity is related to the first passage time above a level, and so as in Theorem 3.12 in \cite{kyprianou2014fluctuations} we have, for $x \leq 0$,
\begin{equation}\label{Thm:Pass:Time:Cont}
\tilde{\E}\left[ e^{-q \tilde{\tau}_{-x}^{+}} ; \tilde{\tau}_{-x}^{+} < \infty \right] = e^{\Phi(q) x}.
\end{equation}
In addition, as in Theorem 2.7 in \cite{Kuznetsov2013} the \textit{$q$-resolvent measure} has a density written, for $x\leq b$, as
\begin{equation}\label{Thm:Res:Dens:Cont}
\tilde{\E}_x \left[ \int_0^{\tilde{\tau}_b^{+}} e^{-qt} \Ind_{\lbrace Y_t \in dy \rbrace} dt \right] =  \left[ e^{-\Phi(q)(b-x)} W^{(q)}(b-y) - W^{(q)}(x-y) \right] dy.
\end{equation}
For the case of periodic observations we have by Theorem 3.1 in \cite{albrecher2016} that, for $\theta \geq 0$,
\begin{equation}\label{Thm:Pass:Time:Per}
\tilde{\E}_x \left[ e^{-q \tilde{T}_a^+ - \theta (Y_{\tilde{T}_a^+} - a)} ; \tilde{T}_a^+ < \infty \right] = \frac{\Phi(\lambda + q) - \Phi(q)}{\Phi(\lambda + q) + \theta} e^{- \Phi(q)(a-x)},\qquad x\leq a,
\end{equation}
and by Theorem 3.1 in \cite{LANDRIAULT2018152}, for $x \leq 0$ and $y \in \mathbb{R}$, 
\begin{equation}\label{Thm:Res:Dens:Per}
\tilde{\E}_x \left[ \int_0^{\tilde{T}_0^+} e^{-qt} \Ind_{\lbrace Y_t \in dy \rbrace} dt \right] =  R^{(q,\lambda)}(x,y) dy,
\end{equation}
where, for $x\leq 0$,
\begin{equation}\label{Res:Dens:PO}
R^{(q,\lambda)}(x,y) :=  \left( \frac{\Phi(\lambda + q) - \Phi(q)}{\lambda} \right) e^{\Phi(q) x} Z^{(q)}(-y;\Phi(\lambda + q)) - W^{(q)}(x-y).
\end{equation} 

\section{Proofs} 
\subsection{Proof of Lemma \ref{Lem:1st:FI}.}
Since $Y=-X$, we can rewrite the expectation in \eqref{1st:Fluc:Id} as
\begin{align*}
\E_x \left[ e^{-q\tau_0^{-} } ; \tau_0^{-} < \infty \right] &= \tilde{\E}_{-x}\left[ e^{-q\tilde{\tau}_0^{+} } ; \tilde{\tau}_0^{+} < \infty \right]= \tilde{\E}\left[ e^{-q\tilde{\tau}_x^{+}} ; \tilde{\tau}_x^{+} < \infty \right].
\end{align*}
Now, by \eqref{Thm:Pass:Time:Cont}, the proof is complete.

\subsection{Proof of Lemma \ref{2nd:Fluc:Id_lemma}.}
First, we use duality to express the expectation in the definition of $g$ as
\begin{align*}
g(x;q,a) = \tilde{\E}_{-x}\left[ \int_0^{\tilde{\tau}_0^+} e^{-qt} \Ind_{ \lbrace Y_t\leq -a \rbrace} dt \right]= \tilde{\E}_{-x}\left[ \int_0^{\tilde{\tau}_0^+} e^{-qt} (1 - \Ind_{ \lbrace Y_t > -a \rbrace } )dt \right] \nonumber.
\end{align*}
We have
\begin{align}
\tilde{\E}_{-x}\left[ \int_0^{\tilde{\tau}_0^+} e^{-qt} dt \right] &= \tilde{\E} \left[ \int_0^{\tilde{\tau}_x^+} e^{-qt} dt \right]=\frac{1}{q} \tilde{\E} \left[ 1- e^{-q \tilde{\tau}_x^+} \right]= \frac{1}{q} \left( 1- e^{-\Phi(q)x} \right). \label{FI:Aux1}
\end{align}
On the other hand, by using the resolvent density as in \eqref{Thm:Res:Dens:Cont}, we obtain that the second integral is given by
\begin{align}
\tilde{\E}_{-x}\left[ \int_0^{\tilde{\tau}_0^+} e^{-qt} \Ind_{ \lbrace Y_t > -a \rbrace } dt \right] &= \tilde{\E} \left[ \int_0^{\tilde{\tau}_x^+} e^{-qt} \Ind_{ \lbrace Y_t > x -a \rbrace} dt \right] \nonumber\\
&= \int_{-\infty}^{x} \Ind_{\lbrace u > x -a \rbrace } \left( e^{-\Phi(q)x} W^{(q)}(x-u) - W^{(q)}(-u) \right) du \nonumber\\
&= e^{-\Phi(q)x} \overline{W}^{(q)}(a) - \overline{W}^{(q)}(a-x). \label{FI:Aux2}
\end{align}
The result follows after combining expressions \eqref{FI:Aux1} and \eqref{FI:Aux2}.
\subsection{Proof of Lemma \ref{1st:FI:Per_lemma}.} \label{proof_1st:FI:Per_lemma}

By the duality relationship between $X$ and $Y$, we have 
\[
\E_x \left[ e^{-q T_0^- + \beta X_{T_0^{-}} } ; T_0^{-} < \infty \right] = \tilde{\E}_{-x} \left[ e^{-q \tilde{T}_0^+ - \beta Y_{\tilde{T}_0^{+}} } ; \tilde{T}_0^{+} < \infty \right].
\]
Now the proof is complete by \eqref{Thm:Pass:Time:Per}.
\subsection{Proof of Proposition \ref{2nd:FI:Per_lemma}.} \label{appendix_2nd:FI:Per_lemma}
(i) Suppose $V_T > 0$.
First, by the definition and using duality, we can write 
\begin{align*}
\Lambda^{(q,\lambda)} (x,z) &= \E_x \left[ \int_0^{T_z^-} e^{-qt} ( 1 - \Ind_{\lbrace X_t < \log V_T \rbrace} ) dt \right] 
= A_1(x ,z) - A_2(x, z),
\end{align*}
where
\begin{align}
A_1(x,z) &:= \E_x \left[ q^{-1} \left( 1-e^{-q T_z^-} \right) \right] \quad \textrm{and} \quad A_2(x,z) := \tilde{\E}_{-x} \left[ \int_0^{\tilde{T}_{-z}^+} e^{-qt} \Ind_{\lbrace Y_t > - \log V_T \rbrace} dt \right]. \nonumber
\end{align}
We can use the identity \eqref{1st:FI:Per} to compute $A_1$ which yields
\begin{align}
A_1(x ,z)
&= q^{-1}\left( 1- J^{(q,\lambda)}(x-z;0) \right).\label{Tax:Aux1}
\end{align}
For $A_2$, we write the expectation in terms of the resolvent density.  By applying  \eqref{Thm:Res:Dens:Per}, we have 
\begin{align*}
A_2(x ,z)= \tilde{\E}_{z-x} \left[ \int_0^{\tilde{T}_{0}^+} e^{-qt} \Ind_{\lbrace Y_t > z - \log V_T \rbrace} dt \right]&=\int_{-z_T}^\infty
R^{(q,\lambda)}(z-x,y)dy
= \int_{-\infty}^{z_T} R^{(q,\lambda)}(z-x,-y)dy, \nonumber
\end{align*}
where $z_T := \log V_T -z$, and for $z-x\leq 0$ the resolvent density is given by \eqref{Res:Dens:PO}. Substituting the resolvent density 
gives
\begin{equation}\label{Tax:Aux:Sub}
\begin{split}
\int_{-\infty}^{z_T} R^{(q,\lambda)}(z-x,-y)dy &= \frac{\Phi(\lambda +q)-\Phi(q)}{\lambda}
e^{\Phi(q)(z-x)}
\int_{-\infty}^{z_T} Z^{(q)}(y;\Phi(\lambda +q))dy \\
&- \int_{-\infty}^{z_T} W^{(q)}(z-x+y) dy.
\end{split}
\end{equation}
By the change of variable,
\begin{align}
\int_{-\infty}^{z_T} Z^{(q)}(y;\Phi(\lambda +q))dy &= \frac{1}{\Phi(\lambda + q)} \left( e^{\Phi(\lambda + q) z_T} - \lambda \int_0^{z_T} e^{-\Phi(\lambda + q)u} W^{(q)}(u) (e^{\Phi(\lambda + q)z_T} - e^{\Phi(\lambda + q)u}) du \right) \nonumber\\
&= \frac{1}{\Phi(\lambda + q)} \left( \lambda \overline{W}^{(q)}(z_T) + Z^{(q)}(z_T; \Phi(\lambda +q)) \right),\label{Tax:Aux2}
\end{align}
and
\begin{equation}\label{Tax:Aux:Conv}
\int_{-\infty}^{z_T} W^{(q)}(z-x+y) dy = \int_{0}^{z-x+z_T} W^{(q)}(u) du=\overline{W}^{(q)}(z-x+z_T) =\overline{W}^{(q)}(\log V_T - x). %
\end{equation}
Substituting \eqref{Tax:Aux2} and \eqref{Tax:Aux:Conv} back into \eqref{Tax:Aux:Sub} yields 
\begin{align}
A_2(x ,z) = \frac{\Phi(\lambda +q)-\Phi(q)}{\lambda \Phi(\lambda + q)}
e^{\Phi(q)(z-x)}
\left( \lambda \overline{W}^{(q)}(z_T) +  Z^{(q)}(z_T; \Phi(\lambda +q)) \right)-  \overline{W}^{(q)}(\log V_T - x) .
 \label{Tax:Res:Final}
\end{align}
Now, the result follows after substracting \eqref{Tax:Res:Final} from \eqref{Tax:Aux1}.

(ii) On the other hand, if $V_T=0$ then $ \Ind_{ \lbrace X_t < \log V_T \rbrace } =  0$ for all $t \geq 0$ and hence
\begin{align*}
\Lambda^{(q,\lambda)}(x,z) =A_1 (x,z)= 
\frac{1}{q} \left( 1 - J^{(q,\lambda)}(x-z;0) \right).
\end{align*}
\bibliographystyle{siam}
\bibliography{CapitalStructure_SpectrallyPositive}
\end{document}